\documentclass[12pt]{amsart}
\usepackage{amssymb}
\usepackage{eucal}
\usepackage{amsmath}
\usepackage{amscd}
\usepackage{txfonts}      
\usepackage[dvips]{color}
\usepackage{mathtools}
\usepackage{multicol}
\usepackage[all]{xy}           
\usepackage{graphicx}
\usepackage{color}
\usepackage{colordvi}
\usepackage{xspace}
\usepackage[normalem]{ulem}
\usepackage{cancel}
\usepackage{enumitem}
\usepackage{tikz}
\usepackage{longtable}
\usepackage{multirow}
\usepackage{ifpdf}
\ifpdf
\usepackage[colorlinks,final,backref=page,hyperindex]{hyperref}
\else
\usepackage[colorlinks,final,backref=page,hyperindex,hypertex]{hyperref}
\fi

\usepackage{comment}

\newtheorem{theorem}{Theorem}[section]
\newtheorem{prop}[theorem]{Proposition}
\newtheorem{lemma}[theorem]{Lemma}
\newtheorem{coro}[theorem]{Corollary}
\newtheorem{thm-def}[theorem]{Theorem-Definition}
\newtheorem{def-prop}[theorem]{Definition-Proposition}
\newtheorem{prop-def}[theorem]{Proposition-Definition}
\newtheorem{coro-def}[theorem]{Corollary-Definition}

\theoremstyle{definition}
\newtheorem{defn}[theorem]{Definition}

\newcommand{\nc}{\newcommand}
\nc{\tred}[1]{\textcolor{red}{#1}}
\nc{\tblue}[1]{\textcolor{blue}{#1}}
\nc{\tgreen}[1]{\textcolor{green}{#1}}
\nc{\tpurple}[1]{\textcolor{purple}{#1}}
\nc{\btred}[1]{\textcolor{red}{\bf #1}}
\nc{\btblue}[1]{\textcolor{blue}{\bf #1}}
\nc{\btgreen}[1]{\textcolor{green}{\bf #1}}
\nc{\btpurple}[1]{\textcolor{purple}{\bf #1}}

\renewcommand{\frak}{\mathfrak}
\newcommand{\efootnote}[1]{}
\renewcommand{\textbf}[1]{}
\newcommand{\delete}[1]{{}}

\delete{
	\nc{\mlabel}[1]{\label{#1} {{\tt {\tiny{(#1)}}}}\ }
	\nc{\mcite}[1]{\cite{#1} {{\tiny\tt (#1)}}\ }
	\nc{\mref}[1]{\ref{#1}{{\tiny\tt (#1)}}\ }
	\nc{\meqref}[1]{~\eqref{#1}{{\tiny\tt (#1)}}\ }
	\nc{\mbibitem}[1]{\bibitem[\bf #1]{#1}}
}

\nc{\mlabel}[1]{\label{#1}}  
\nc{\mcite}[1]{\cite{#1}}  
\nc{\mref}[1]{\ref{#1}}  
\nc{\meqref}[1]{~\eqref{#1}}
\nc{\mbibitem}[1]{\bibitem{#1}} 

\nc{\sbar}{, }

\nc{\wvec}[2]{{\scriptsize{ [
			\begin{array}{c} #1 \\ #2 \end{array}   ]}}}
\nc{\lp}{\big ( }
\nc{\llp}{\Big (}
\nc{\Llp}{\left (}
\nc{\rp}{\big ) }
\nc{\rrp}{\Big )}
\nc{\Rrp}{\right )}
\nc{\lb}{\big < }
\nc{\llb}{\!\Big \langle }
\nc{\Llb}{\! \left <}
\nc{\rb}{\big >  }
\nc{\rrb}{\Big \rangle \!}
\nc{\Rb}{\Big \rangle\! }
\nc{\length}{{\rm leng}}

\nc{\bin}[2]{ (_{\stackrel{\scs{#1}}{\scs{#2}}})}  
\nc{\binc}[2]{ \Big (\! \begin{array}{c} \scs{#1}\\
		\scs{#2} \end{array}\! \Big )}  
\nc{\bincc}[2]{  \left ( {\scs{#1} \atop
		\vspace{-1cm}\scs{#2}} \right )}  
\nc{\bs}{\bar{S}}
\nc{\cosum}{\sqsubset}
\nc{\la}{\longrightarrow}
\nc{\rar}{\rightarrow}
\nc{\dar}{\downarrow}
\nc{\dap}[1]{\downarrow \rlap{$\scriptstyle{#1}$}}
\nc{\uap}[1]{\uparrow \rlap{$\scriptstyle{#1}$}}
\nc{\defeq}{\stackrel{\rm def}{=}}
\nc{\disp}[1]{\displaystyle{#1}}
\nc{\dotcup}{\ \displaystyle{\bigcup^\bullet}\ }
\nc{\gzeta}{\bar{\zeta}}
\nc{\hcm}{\ \hat{,}\ }
\nc{\hts}{\hat{\otimes}}
\nc{\barot}{{\otimes}}
\nc{\free}[1]{\bar{#1}}
\nc{\uni}[1]{\tilde{#1}}          
\nc{\hcirc}{\hat{\circ}}
\nc{\lleft}{[}
\nc{\lright}{]}
\nc{\curlyl}{\left \{ \begin{array}{c} {} \\ {} \end{array}
	\right .  \!\!\!\!\!\!\!}
\nc{\curlyr}{ \!\!\!\!\!\!\!
	\left . \begin{array}{c} {} \\ {} \end{array}
	\right \} }
\nc{\longmid}{\left | \begin{array}{c} {} \\ {} \end{array}
	\right . \!\!\!\!\!\!\!}
\nc{\ora}[1]{\stackrel{#1}{\rar}}
\nc{\ola}[1]{\stackrel{#1}{\la}}
\nc{\ot}{\otimes}
\nc{\mot}{{{\sbar}}}
\nc{\otm}{\mot}
\nc{\scs}[1]{\scriptstyle{#1}}
\nc{\subv}{{^{\star}}}
\nc{\cov}{{^{\sharp}}}
\nc{\mrm}[1]{{\rm #1}}
\nc{\dirlim}{\displaystyle{\lim_{\longrightarrow}}\,}
\nc{\invlim}{\displaystyle{\lim_{\longleftarrow}}\,}
\nc{\proofbegin}{\noindent{\bf Proof: }}
	\nc{\proofend}{$\quad \square$ \vspace{0.3cm}}
\nc{\sha}{{\mbox{\cyr X}}}  
\nc{\shap}{{\mbox{\cyrs X}}} 
\nc{\shpr}{\diamond}    
\nc{\shplus}{\shpr^+}
\nc{\shprc}{\shpr_c}    
\nc{\msh}{\ast}
\nc{\vep}{\varepsilon}
\nc{\labs}{\mid\!}
\nc{\rabs}{\!\mid}

\newcommand{\C}{\mathbb{C}}

\newcommand{\Q}{\mathbb{Q}}
\newcommand{\R}{\mathbb{R}}

\newcommand{\Z}{\mathbb{Z}}

\newcommand {\calh}{{\mathcal {H}}}

\newcommand {\calp}{{\mathcal {P}}}

\nc{\fraka}{{\frak a}}
\nc{\frakA}{{\frak A}}
\nc{\frakb}{{\frak b}}
\nc{\frakB}{{\frak B}}
\nc{\frakf}{{\frak F}}
\nc{\frakh}{{\frak h}}
\nc{\frakH}{{\frak H}}
\nc{\frakk}{{\frak k}}
\nc{\frakK}{{\frak K}}
\nc{\frakM}{{\frak M}}
\nc{\frakm}{{\frak m}}
\nc{\frakP}{{\frak P}}
\nc{\frakp}{{\frak p}}
\nc{\frakS}{{\frak S}}
\nc{\dep}{\mathrm{dep}}
\nc{\bfrakM}{\overline{\frakM}}

\nc{\vsa}{\vspace{-.1cm}}
\nc{\vsb}{\vspace{-.2cm}}
\nc{\vsc}{\vspace{-.3cm}}
\nc{\vsd}{\vspace{-.4cm}}
\nc{\vse}{\vspace{-.5cm}}

\nc {\e} {{\epsilon}}
\nc{\fpower}{\calp_{\rm fin}}
\nc{\pfpair}[2]{\Big(\begin{array}{c}\scs{#1} \\ \scs{#2} \end{array} \Big)}

\font\cyr=wncyr10
\font\cyrs=wncyr7
\nc{\zb}[1]{\textcolor{blue}{Bin: #1}}
\nc{\xhy}[1]{\textcolor{red}{Yu: #1}}
\nc{\lir}[1]{\textcolor{purple}{Li:#1}}

\newcommand {\calha}{{{\mathcal {H}} _{\Z }}}

\newcommand {\calhc}{{{\mathcal {H}} _{\Z _{\le 0}}}}
\newcommand {\calhd}{{{\mathcal {H}} _{\Z _{\ge 1}}}}

\newcommand {\calhf}{{{\mathcal {H}}^* _{\Z _{\le 0}}}}
\newcommand {\calhg}{{{\mathcal {H}}^* _{\Z _{\ge 1}}}}

\newcommand {\Deltaa} {{\Delta_{\ge 1}}}
\newcommand {\Deltab} {\Delta _{\le0}}
\newcommand {\Deltac} {\tilde  \Delta _{\ge 1}}
\newcommand {\Deltad} {\tilde  \Delta _{\le0}}
\newcommand {\Deltae} {\Delta^*_{\ge 1}}
\newcommand {\Deltaf} {\Delta^*_{\le 0}}
\nc {\cks}{\text{\textcircled {s}}\xspace}
\nc{\id}{\mathrm{id}}

\nc{\rdual}{Riemann-type dual map\xspace}

\begin{document}
	\title[A Riemann-type duality of shuffle Hopf algebras]{A Riemann-type duality of shuffle Hopf algebras\\
		related to multiple zeta values }
	%

\author{Li Guo}
\address{Department of Mathematics and Computer Science,
Rutgers University, Newark, NJ 07102, USA}
\email{liguo@rutgers.edu}

\author{Hongyu Xiang}
\address{School of Mathematics,
Sichuan University, Chengdu, 610064, P. R. China}
\email{mathxhy@aliyun.com}

\author{Bin Zhang}
\address{School of Mathematics,
Sichuan University, Chengdu, 610064, P. R. China}
\email{zhangbin@scu.edu.cn}

\date{\today}

\begin{abstract}
This paper offers a Hopf algebraic interpretation of a functional equation of multiple zeta functions, motivated by the classical symmetry of the Riemann zeta function. Starting from the extended shuffle algebra that encodes multiple zeta values (MZVs) at integer arguments, we show that its subalgebra corresponding to nonpositive arguments carries a natural differential Hopf algebra structure. This Hopf algebra is in graded linear duality with the shuffle Hopf algebra associated to MZVs at positive arguments. The resulting duality, realized through an explicit isomorphism, provides an algebraic analog of the functional equation relating $\zeta(s)$ with $\zeta(1-s)$ of the Riemann zeta function and unifies the positive and nonpositive sectors of multiple zeta functions within a common Hopf algebraic framework.
\end{abstract}

\subjclass[2020]{
	11M32,	
	16T05,	
	12H05, 
	16T30,  
	16W25,	
	16S10,	
	40B05	
}
\keywords{multiple  zeta function, shuffle product, functional equation, Hopf algebra, duality, derivation}

\maketitle

\vspace{-1cm}

	\tableofcontents

\vspace{-1cm}

\allowdisplaybreaks

\setcounter{section}{0}

\section{Introduction}
Motivated by generalizing the functional equation of the Riemann zeta function to multiple zeta functions, this paper gives an algebraic analog as an isomorphism between two differential Hopf algebras, one for the shuffle algebra encoding multiple zeta values with positive arguments and one corresponding to multiple zeta values with nonpositive arguments.

\subsection{Functional equations and functional relations for multiple zeta values}
	The multiple zeta series (MZS)
	\begin {equation}
	\mlabel{eq:MultipleZetaSeries}
\zeta(s_1,\cdots,s_k)=	\sum _{n_1>\cdots > n_k>0}\frac 1{n_1^{s_1}\cdots n_k^{s_k}}
\end{equation}
is a natural generalization of the Riemann zeta series
\begin {equation}
	\mlabel {eq:RiemannZetaSeries}
\zeta(s)=	\sum _{n>0}\frac 1{n^{s}}.
\end{equation}
They attract great attention since the 1990s because their values
at convergent points with positive integer arguments, called {\bf multiple zeta values} (MZVs), have appeared in broad areas of mathmetics and mathematical physics, such as arithmetic geometry, quantum field theory, quantum groups and knot theory\mcite{BBBL, BK, GM, GZ2, H3, IKZ, Kre, Ter, Zag}.

This series converges in the region
\begin{equation}
\mlabel {eq:region}
\{\vec s \in \C ^k \ |\ \Re (s_1)+\cdots + \Re (s_i)>i, \ i=1, \cdots , k\},
\end{equation}
thus defining a holomorphic function in this region. It is shown \mcite{AET,Zj} that this holomorphic function can be analytically extended to a meromorphic function $\zeta (\vec s)$ in $\C ^k$, called the {\bf multiple zeta function} on $\C ^k$, with simple poles at
\vsa
\begin{equation*}
s_1=1; {\rm or } \ s_{1}+s_2=2,1,0,-2,-4, \cdots; {\rm\ or \ } \sum_{i=1}^j s_{i} \in \Z_{\leq j}\ (3\leq j\leq k).
\vsa
\end{equation*}

It is well-known that the Riemann zeta function satisfies the functional equation:
\begin{equation}
\zeta (1-s)=2(2\pi )^{-s}\Gamma (s)\cos (\frac {\pi s}2)\zeta (s).
\end{equation}
So a natural question about multiple zeta functions is the existence of their functional equations. To address this challenging question, it is helpful to borrow some insight from the case of the Riemann zeta function, for which the process to derive the functional equation can be divided into two stages.
\begin{enumerate}
 \item[\bf{ Stage 1.}$\!\!\!$] \ \  Desingularize the meromorphic function.
  \item[{\bf Stage 2.}$\!\!\!$] \ \ Detect symmetries of the desingularized function.
\end{enumerate}
For multiple zeta functions, the Stage 1 is achieved in \mcite {FKMT}. The shuffle product of two MZVs is given as the Euler's decomposition formula for the product of two Riemann zeta values~\mcite{BBG} and is given in~\mcite{GX} in the general case. This type of shuffle product formula is extended to desingularized multiple zeta functions and multiple zeta values, and is presented as a functional relation in\, \mcite{Ko,Ko2,SV}.
This functional relation can be regarded as an interpretation of Stage 2.

\vsa
\subsection{A duality of Hopf algebras for MZVs}
This paper takes a Hopf algebra approach to a functional equation of multiple zeta functions.

\subsubsection{Hopf algebras in the study of MZVs}
Hopf algebras have played a critical role in the study of MZVs.
Lifting the (cyclotomic) MZVs to the motivic level, motivic (cyclotomic) MZVs form a Hopf algebra.
This structure reveals further properties of MZVs, such as depth filtration and decompositions of MZVs, leading to the seminal works of Brown, Goncharov and others~\mcite{Bn,DG,Gon,GM}.

On a more concrete level, the shuffle algebra and quasi-shuffle algebra from encoding MZVs have natural Hopf algebras. Such Hopf algebra structures allow the Connes-Kreimer approach to renormalization in quantum field theory\,\mcite{CK} to be applied to studying the multiple zeta series at divergent arguments~(see e.g. \mcite{GZ2}).
Recently, a differential Hopf algebra structure was obtained on the shuffle Hopf algebra from MZVs~\mcite{GHXZ1}.

The purpose of this paper is to provide a Hopf algebraic interpretation of the two stages for a functional equation of multiple zeta functions. Here for Stage 1, the desingularization is on an algebraic level in terms of an extension of the shuffle algebra of MZVs. For Stage 2, the symmetric is achieved as a duality of two Hopf algebras, one reflects MZVs with positive arguments, and one from the extended shuffle algebra, in analog to MZVs with nonpositive arguments.

Before giving more details, we fix some notations.
For the set $X=\Z$, $\Z _{\ge 1}$, or $\Z _{\le 0}$, denote
\begin{equation}
\calh_X^{+}:=\bigoplus\limits _{k\in \Z _{>0}, \vec s\in X ^k }\Q [\vec s],\quad
\calh _X:=\Q {\bf 1}\oplus \bigoplus\limits _{k\in \Z _{>0}, \vec s\in X ^k }\Q [\vec s].
\mlabel{eq:hopf}
\end{equation}
Here $[\vec s]$ is only a symbol without the algebraic structures inherited from ${\vec s}\in X^k$.

\subsubsection{Shuffle algebra of MZV $\calhd$}
Our starting point is the shuffle algebra of MZVs. Recall that the structure of regularized MZVs with positive arguments \mcite{IKZ}, namely $\zeta(s_1,\ldots,s_k)$ with $s_i\geq 1$ for $1\leq i\leq k$, can be encoded by the MZV shuffle algebra
$(\calhd,\shap_{\geq 1})$. The encoding is achieved by
the linear map between the underlying spaces of a noncommutative polynomial algebra and a commutative algebra
$$ \Q \langle x_0, x_1 \rangle x_1 \longrightarrow \R[T], $$
stemming from using $x_0^{s_1-1}x_1\cdots x_0^{s_k-1}x_1$ to represent the integral expression of the MZV~\mcite{Zag}:
{\small
\begin {equation}
\mlabel {eq:zetai}
\zeta (s_1, \cdots s_k)=\underbrace {\int_ 0^1\frac {dt}t\int _0^t \frac {dt}t\cdots \int _0^t\frac {dt}t}_{s_1-1}\int _0^t\frac {dt}{1-t}\cdots \underbrace {\int _0^t\frac {dt}t\cdots \int _0^t\frac {dt}t}_{s_k-1}\int _0^t\frac {dt}{1-t},
\end{equation}
}
followed by the Ihara-Keneko-Zagier regularization sending $x_1$ to $T$. Here the subscripts of the variable $t_i, 1\leq i\leq s_1+\cdots+s_k,$ are suppressed for simplicity.
The shuffle product of two regularized MZVs is encoded as the usual shuffle product of the corresponding two words in $\Q \langle x_0, x_1 \rangle x_1$ with $1$ as the identity.
Then the shuffle product $\shap_{\geq 1}$ on $\calhd$ is defined by pulling back the shuffle product on $\Q\langle x_0,x_1\rangle x_1$ via the linear bijection
\vsa
\begin {equation*}
\rho: \calhd \to \{1\}\cup \Q \langle x_0, x_1 \rangle x_1, \ {\bf1}\mapsto 1, [s_1, \cdots, s_k]\mapsto x_0^{s_1-1}x_1\cdots x_0^{s_k-1}x_1.
\end{equation*}
As noted above, an explicit formula of $\shap_{\geq 1}$ is given in \mcite{GX}, generalizing Euler's Decomposition Formula. Going forward, we will call $(\calhd,\shap_{\geq 1})$ the {\bf shuffle algebra of MZVs}.
Together with the quasi-shuffle (stuffle) algebra, it provides the algebraic framework to study algebraic relations among MZVs\,\mcite{H2,H3,IKZ}.
\vsa
\subsubsection{Extended shuffle algebra $\calha$ and its subalgebras}
In two recent articles ~\mcite{GHXZ1,GHXZ2}, this shuffle algebra of MZVs is expanded in two directions. In \mcite{GHXZ2}, the above shuffle product $\shap_{\geq 1}$ on $\calhd$ is extended to a product on $\calha$ that is characterized uniquely by a differential recursion (see Theorem~\mref{thm:Xshap}).
Since the canonical basis of $\calha$ shown in Eq.\meqref{eq:hopf}
parameterizes (formal) special values of multiple zeta functions,
in the spirit of the encoding of MZVs by the shuffle algebra of MZVs, this extended shuffle algebra can be regarded as an algebraic interpretation of the analytic continuation of multiple zeta functions.

On the other hand, in \mcite{GHXZ1}, with a coproduct $\Delta_{\geq 1}$ again characterized by a (shifted) coderivation recursion, the shuffle algebra $\calhd$ is enriched to a Hopf algebra $(\calhd,\shap_{\geq 1}, \Delta_{\ge 1})$. See Proposition~\mref{pp:unique} for details.
Other than the subalgebra $\calhd$, the extended shuffle algebra $\calha$ contains another subalgebra $\calhc$ (Proposition\,\mref{lem:le0differentialA}).

\subsubsection{Duality of Hopf algebras}
The first goal of this paper is to show that there is also a natural Hopf algebra structure on $\calhc$ (Theorem\,\mref{prop:calhcHopf}), as an interpretation of Stage 1 in the context of Hopf algebras. Unlike the shuffle Hopf algebra of MZVs, this new Hopf algebra is not commutative, though still cocommutative.

The second goal is to show that there is a perfect pairing between the two Hopf algebras $\calhd$ and $\calhc$, achieved by a Hopf algebra isomorphism
\vsa
$$\varphi:\big(\calhd,\shap_{\ge 1},\Delta_{\ge 1}\big) \longrightarrow \big(\calhf,\Delta_{\le 0}^*,\shap_{\le 0}^*\big).
\vsa
$$
In addition, the derivation introduced in \mcite{GHXZ2} that characterizes the extended shuffle product on $\calha$ also equipped the two Hopf algebras with a differential Hopf algebra structure which is preserved by $\varphi$.
The relations among the four differential Hopf algebras are summarized in the diagram
\vsb
\begin{equation} \mlabel{eq:diag}
	\begin{split}
\xymatrix@1{ {\ \ \ } \ar@{<->}_{\txt{\small linear \\ \small duality}}[d] &\calhc \ar^(.35){\varphi^*}[drr] && \calhd \ar_(.35){\varphi}[dll]  \\
& \calhf && \calhg
}
\end{split}
\vsa
\end{equation}
Here the vertical pairs are in linear dual; while the diagonal pairs are isomorphic under the \rdual $\varphi$ and its linear dual $\varphi^*$. See Theorem~\mref{t:dhaiso} for the precise statement.
\vse
\subsection{Outline of the paper} Thus the purpose of this paper is to algebraically study the symmetry of the multiple zeta function $\zeta (\vec s)$ at $\vec s \in \Z ^k$, as a multi-variable algebraic analog of the functional equation for the Riemann zeta function.

In Section \mref{s:diff}, we recall some structures related to MZVs, which are needed for our study: the extended shuffle product $\shap$ on $\calha$ and the graded Hopf algebra  $(\calhd, \shap_{\geq 1}, \Deltaa)$. A differential operator $J _{\ge 1}$ is also introduced to make $(\calhd, \shap_{\ge 1}, \Deltaa, J_{\ge 1})$ into a differential Hopf algebra (Theorem\,\mref{thm:ge1diffHopf}). In parallel with this, for the subalgebra  $(\calhc, \shap_{\le 0})$, we construct a compatible coproduct $\Delta_{\le 0}$ via the differential operator  $J_{\le 0}$, thereby forming a second differential Hopf algebra (Theorem\,\mref{thm:le0diffHopf}). The Hopf algebras from their graded linear duals are also described.

Section \mref{s:iso} is devoted to the study of the duality between the above-introduced differential Hopf algebras. In Section~\mref{ss:main},
the {\bf \rdual} of Hopf algebras
\vsa
\begin{equation}
	\varphi:\calhd\to \calhf, \quad
\varphi([s_1,\cdots, s_k]):=[1-s_1, \cdots, 1-s_k]^{\ast},
\mlabel{eq:dual}
\vsa
\end{equation}
is introduced and is shown to give an isomorphism of differential Hopf algebras (Theorem~\mref{t:dhaiso}). The facts that the duality map is indeed an algebra isomorphism and a coalgebra isomorphism are verified in Sections \mref{ss:coisocg} and \mref{ss:coisodf} respectively.
\vsd
\section{Differential Hopf algebras}
\mlabel{s:diff}
In this section, we first review the extended shuffle algebra \cite{GHXZ2} and the Hopf algebra arising from multiple zeta values \cite{GHXZ1}. Next, we equip this Hopf algebra with a compatible differential operator $J_{\ge 1}$, thus forming a differential Hopf algebra. We also construct a differential Hopf algebra on the subalgebra of negative integers by using the restriction $J_{\le 0}$ of the differential operator $J$ to $\calhc$. Finally, we present the graded dual structures of these two differential Hopf algebras.
\vsc
\subsection{The extended shuffle algebra $\calha$ and its subalgebras}
We recall the extended shuffle algebra $\calha$ and give its two subalgebras.
\vsc
\subsubsection{The extended shuffle algebra $\calha$}

Following the notation in Eq.~\meqref{eq:hopf}, define a linear map
\vsa
\begin{equation*}
I:\calh_{\Z}^{+}  \longrightarrow  \calh_{\Z}^{+}, \quad [s_1, s_2, \cdots, s_k]\mapsto [s_1+1, s_2, \cdots, s_k],
\vsa
\end{equation*}
and define $J:\calha \to \calha$ to be the linear map given by the inverse of $I$ on $\calh^+ _{\Z}$ and by $J({\bf 1})=0:$
\vsa
\begin{equation}
J:\calha \to \calha: {\bf 1}\mapsto 0, \quad [s_1, s_2,\cdots, s_n]\mapsto [s_1-1,s_2, \cdots, s_n].
\mlabel{eq:j}
\vsa
\end{equation}

Define the {\bf depth} of $[\vec s]$ for $\vec s\in \Z^n$ to be $\dep([\vec s]):=n$. Then $\calha$ becomes a graded vector space. 
Then the extended shuffle product on $\calha$ is characterized by the following property.
\vsb
\begin{theorem} \mcite {GHXZ2}
\mlabel{thm:Xshap}
With respect to the depth grading, there is a unique graded  associative product $\shap$ on $\calha$, called the {\bf extended shuffle product}, with the unit ${\bf 1}$, such that
\begin{enumerate}
\item $[0]\shap [\vec s]=[0,\vec s]$ for $\vec s\in \Z^k;$
\item $[\vec s]\shap [0]=[0, \vec s]$ for $\vec s\in\Z_{>0}\times \Z^{k-1};$
\item $(\calha,\shap,J)$ is a {\bf differential algebra} in the sense that $J$ is a {\bf derivation}, satisfying the {\bf Leibinz rule}$:$
$J(u\shap v)=J(u)\shap v + u\shap J(v), u, v\in \calha.$
\end{enumerate}
\end{theorem}

We have the following generalization of the Euler decomposition formula \mcite{BBG,GX}.
\begin {lemma}
\mlabel{lemma:ShuffleInHGe1}
For $(s_1, \vec s\,')\in \Z_{\ge 1}\times \Z^k, (t_1, \vec t\,')\in \Z_{\ge 1}\times \Z^{\ell}$, $k, \ell\in \Z_{\ge 1}$,
\begin{equation}
[s_1, \vec s\,']\shap [t_1, \vec t\,']
=\sum_{i=0}^{t_1-1} \binc{s_1-1+i}{i} \Big[s_1+i, \vec s\,'\shap [t_1-i, \vec t\,']\Big]+\sum_{i=0}^{s_1-1}\binc{t_1-1+i}{i}\Big[t_1+i, [s_1-i, \vec s\,']\shap \vec t\,'\Big].
\mlabel{eq:GenEuler}
\end{equation}
\end{lemma}
\begin{proof}
We prove the statement by induction on $m:=s_1+t_1\ge 2$ where $s_1,t_1\geq 1$. For the initial step $m=2$, that is $s_1=t_1=1$, by the definition of extended shuffle product we have
$$
[1, \vec s\, ']\shap [1, \vec t\,']=\Big[1, \vec s\,'\shap [1, \vec t\,']\Big]+\Big[1, [1, \vec s\,']\shap \vec t\,'\Big].
$$
For given $n\geq 2$, assume that Eq.\,(\mref{eq:GenEuler}) holds for $m=s_1+t_1=n$, and we prove Eq.\,(\mref{eq:GenEuler}) for $m=n+1$ by dividing into three cases. \\
{\bf Case 1:} If $s_1=1$, $t_1>1$, then
\vsb
\begin{equation*}
[1, \vec s\,']\shap [t_1, \vec t\,']=\Big[1, \vec s\,'\shap [t_1, \vec t\,']\Big]+I\Big([1,\vec s\,']\shap [t_1-1, \vec t\,']\Big).
\end{equation*}
By the induction hypothesis,
\vsb
$$[1,\vec s\,']\shap [t_1-1, \vec t\,']=\sum_{i=0}^{t_1-2} \Big[i+1, \vec s\,'\shap [t_1-1-i, \vec t\,']\Big]+\Big[t_1-1, [1, \vec s\,']\shap \vec t\,'\Big].
\vsc
$$
Thus
\vsb
\begin{equation*}
\begin{split}
[1, \vec s\, ']\shap [t_1, \vec t\,']&=\Big[1, \vec s\,'\shap [t_1, \vec t\,']\Big]+I\Big(\sum_{i=0}^{t_1-2} \Big[i+1, \vec s\,'\shap [t_1-1-i, \vec t\,']\Big]+\Big[t_1-1, [1, \vec s\,']\shap \vec t\,'\Big]\Big)\\
&=\Big[1, \vec s\,'\shap [t_1, \vec t\,']\Big]+\sum_{i=0}^{t_1-2} \Big[i+2, \vec s\,'\shap [t_1-1-i, \vec t\,']\Big]+\Big[t_1, [1, \vec s\,']\shap \vec t\,'\Big]\\
&=\sum_{i=0}^{t_1-1} \Big[i+1, \vec s\,'\shap [t_1-i, \vec t\,']\Big]+\Big[t_1, [1, \vec s\,']\shap \vec t\,'\Big],
\end{split}
\end{equation*}
which is Eq.\,(\mref{eq:GenEuler}) for the special case $s_1=1$, $t_1>1$.\\
{\bf Case 2:} If $s_1>1$, $t_1=1$, then the proof is similar to Case 1.\\
{\bf Case 3:} If $s_1>1$, $t_1>1$, then
\vsb
\begin{align*}
[s_1, \vec s\,']\shap [t_1, \vec t\,']&=I\Big([s_1-1, \vec s\,']\shap [t_1, \vec t\,']+[s_1, \vec s\,']\shap [t_1-1, \vec t\,']\Big).
\vsa
\end{align*}
By the inductive hypothesis, we have
\vsb
\begin{align*}
[s_1-1, \vec s\,']\shap [t_1, \vec t\,']=&\sum_{i=0}^{t_1-1} \binc{s_1-2+i}{i} \Big[s_1-1+i, \vec s\,'\shap [t_1-i, \vec t\,']\Big]
+\sum_{i=0}^{s_1-2}\binc{t_1-1+i}{i}\Big[t_1+i, [s_1-1-i, \vec s\,']\shap \vec t\,'\Big],\\
[s_1, \vec s\,']\shap [t_1-1, \vec t\,']=&\sum_{i=0}^{t_1-2} \binc{s_1-1+i}{i} \Big[s_1+i, \vec s\,'\shap [t_1-1-i, \vec t\,']\Big]
+\sum_{i=0}^{s_1-1}\binc{t_1-2+i}{i}\Big[t_1-1+i, [s_1-i, \vec s\,']\shap \vec t\,'\Big].
\end{align*}
\vsa
Hence,
\vsb
\begin{align*}
[s_1, \vec s\,']\shap [t_1, \vec t\,']=&\sum_{i=0}^{t_1-1} \binc{s_1-2+i}{i} \Big[s_1+i, \vec s\,'\shap [t_1-i, \vec t\,']\Big]\\
&+\sum_{i=0}^{s_1-2}\binc{t_1-1+i}{i}\Big[t_1+1+i, [s_1-1-i, \vec s\,']\shap \vec t\,'\Big]\\
&+\sum_{i=0}^{t_1-2} \binc{s_1-1+i}{i} \Big[s_1+1+i, \vec s\,'\shap [t_1-1-i, \vec t\,']\Big]\\
&+\sum_{i=0}^{s_1-1}\binc{t_1-2+i}{i}\Big[t_1+i, [s_1-i, \vec s\,']\shap \vec t\,'\Big]\\
=&\Big[s_1, \vec s\,'\shap [t_1, \vec t\,']\Big]+\sum_{i=0}^{t_1-2}\Big(\binc{s_1-1+i}{i+1}+\binc{s_1-1+i}{i}\Big)\Big[s_1+i+1, \vec s\,'\shap [t_1-1-i, \vec t\,']\Big]\\
&+\Big[t_1, [s_1, \vec s\,']\shap \vec t\,'\Big]+\sum_{i=0}^{s_1-2}\Big(\binc{t_1-1+i}{i}+\binc{t_1-1+i}{i+1}\Big)\Big[t_1+1+i, [s_1-1-i, \vec s\,']\shap \vec t\,'\Big]\\
\overset{\textcircled{1}}{=}&\sum_{i=0}^{t_1-1} \binc{s_1-1+i}{i} \Big[s_1+i, \vec s\,'\shap [t_1-i, \vec t\,']\Big]+\sum_{i=0}^{s_1-1}\binc{t_1-1+i}{i}\Big[t_1+i, [s_1-i, \vec s\,']\shap \vec t\,'\Big].
\vsa
\end{align*}
Here $\textcircled{1}$ follows from $\binc{n}{i}+\binc{n}{i+1}=\binc{n+1}{i+1}$. Therefore Eq.\,(\mref{eq:GenEuler}) is proved.
\end{proof}

\subsubsection{The subalgebra $\calhd$}
It follows from \mcite{GZ1} that $\calhd$ is a subalgebra  of  $(\calha, \shap)$. Let $\shap_{\ge 1}$ denote this restriction of $\shap$. But $\calhd$ is not closed under the action of $J$. We modify $J$ to an operator $J_{\ge 1}$ on $\calhd$ as follows.
\vsc
{\small
\begin{equation}\label{eq:Jge1}
J_{\ge 1}:\calhd \longrightarrow \calhd,\quad
\left \{ \begin{array}{ll}
{\bf 1}&\mapsto 0,\\
{ } [s_1, \vec s]&\mapsto \left\{ \begin{array}{ccc}
[s_1-1, \vec s], &s_1>1,\\
0, &s_1=1.
\end{array}\right.
\end{array}\right .
\end{equation}
}
Furthermore, the algebra $\calhd$ has the grading by {\bf weight} $|\vec s|:=s_1+\cdots+s_k$, that is,
\vsb
\begin{equation}
\begin{split}
\calhd=\bigoplus_{m=0}^{\infty}\Q H_m, \text{ with } \left\{\begin{array}{l} H_0:={\bf 1},\\
	H_m:=\big\{[s_1,\cdots,s_k]\in \Z_{\ge 1}^k\,\big|\, s_1+\cdots+s_k=m, k\geq 1\big\}.\label{eq:Hge1grading}
\end{array} \right .
\end{split}
\end{equation}
Then we have 
\begin {lemma}
\mlabel{lem:Jge1der} The triple $(\calhd, \shap_{\ge 1}, J_{\ge 1})$ is a graded differential algebra.
\end{lemma}

\begin {proof}
It is enough to verify the Leibniz rule for basis elements in $\calh^+_{\Z_{\ge 1}}$. 
 For  $(s_1,\vec s\,')\in \Z_{\ge 1}^k$ and $(t_1, \vec t\,')\in \Z_{\ge 1}^{\ell}$, we distinguish the following four cases.

\smallskip
\noindent {\bf Case 1:}  If $s_1>1$, $t_1>1$, then
\vsb
\begin{align*}
 &J_{\ge 1}\Big([s_1,\vec s\,']\shap_{\ge 1} [t_1, \vec t\,']\Big)=J_{\ge 1}\Big(I([s_1-1, \vec s\,'])\shap_{\ge 1} I([t_1-1, \vec t\,'])\Big)\\
=&J_{\ge 1}I\Big([s_1-1, \vec s\,']\shap_{\ge 1} [t_1, \vec t\,']+[s_1, \vec s\,']\shap_{\ge 1} [t_1-1, \vec t\,']\Big)\\
=&[s_1-1, \vec s\,']\shap_{\ge 1} [t_1, \vec t\,']+[s_1, \vec s\,']\shap_{\ge 1} [t_1-1, \vec t\,'],
\end{align*}
which equals to $\Big(J_{\ge 1}([s_1,\vec s\,'])\Big)\shap_{\ge 1} [t_1, \vec t\,']+[s_1, \vec s\,']\shap_{\ge 1} \Big( J_{\ge 1}([t_1, \vec t\,'])\Big)$.
Here the second equality holds by the fact that $I$  is a Rota-Baxter operator of weight $0$\,\mcite{GZ1}, the third equality holds by $J_{\ge 1}I=\id$.

\smallskip
\noindent
{\bf Case 2:}  If $s_1=1$, $t_1>1$, since $J_{\ge 1}I=\id$ on $\calhd$, then by the definition of $\shap_{\ge 1}$, we have
\vsb
  \begin{equation*}
  \begin{split}
&J_{\ge 1}\Big([1,\vec s\,']\shap_{\ge 1} [t_1, \vec t\,']\Big)=J_{\ge 1}\Big(\Big[1, \vec s\,'\shap_{\ge 1} [t_1, \vec t\,']\Big]\Big)+J_{\ge 1} I\Big([1, \vec s\,']\shap_{\ge 1} [t_1-1, \vec t\,']\Big)\\
=&0+[1, \vec s\,']\shap_{\ge 1} [t_1-1, \vec t\,']=[1, \vec s\,']\shap_{\ge 1}  [t_1-1, \vec t\,'].
\end{split}
\end{equation*}
This agrees with $\Big(J_{\ge 1}([1,\vec s\,'])\Big)\shap_{\ge 1} [t_1, \vec t\,']+[1, \vec s\,']\shap_{\ge 1}  \Big(J_{\ge 1}([t_1, \vec t\,'])\Big)$ since the first term vanishes.

\smallskip
\noindent
{\bf Case 3:} If $s_1>1$, $t_1=1$. The proof is similar to Case 2.

\smallskip
\noindent
{\bf Case 4:} If $s_1=1$, $t_1=1$, then on the one hand,
\vsa
\begin{align*}
J_{\ge 1}([1,\vec s\,']\shap_{\ge 1} [1, \vec t\,'])= J_{\ge 1}\Big(\Big[1, \vec s\,'\shap_{\ge 1} [1, \vec t\,']\Big]+\Big[1, [1,\vec s\,']\shap_{\ge 1} \vec t\,'\Big]\Big)=0.
\end{align*}
On the other hand, we have
\vsa
$$\Big(J_{\ge 1}([1,\vec s\,'])\Big)\shap_{\ge 1} [1, \vec t\,']+[1,\vec s\,']\shap_{\ge 1} \Big(J_{\ge 1}( [1, \vec t\,'])\Big)=0.
$$

Additionally,
$\shap_{\ge 1}$ is known to preserve the grading.
By Eq.\,\meqref{eq:Jge1}, $J_{\ge 1}$ is of degree $-1$.
Now the proof is completed.
\end{proof}
\vsb
\subsubsection{The subalgebra $\calhc$}

By \cite{GHXZ2}, the subspace $\calhc $ is closed under the extended shuffle product $\shap$ and stable under the linear operator $J: \calha\longrightarrow \calha$.
For distinction, we will use $\shap_{\le 0}$ and $J_{\le 0}$ to denote these restrictions of $\shap$ and $J$ respectively.
With these notations, we have
\begin{prop}[\mcite {GHXZ2}]
	\mlabel{lem:le0differentialA}
	The triple $(\calhc, \shap_{\le 0}, J_{\leq 0})$  is  a  differential subalgebra  of  $(\calha, \shap, J)$.
\end{prop}
Denote
\vsb
$$G_0:=\{{\bf1}\}, \ \ G_{k}:=\{[\vec s] \ |\ (s_1,\cdots,s_n)\in \Z_{\le 0}^n,\ n-(s_1+\cdots+s_n)=k\}, k\ge 1.
$$
Then $G_k$ is finite for each $k\in \Z _{\ge 0}$, and  $(\calhc,\shap_{\le 0})$ is a graded algebra with the grading
\vsa
\begin {equation}
\mlabel {eq:grade<1}
\calhc=\bigoplus_{k\ge 0} \Q G_k.
\vsb
\end{equation}
\begin{lemma}
\mlabel {lem:suba}
The subalgebra $(\calhc, \shap_{\le 0})$ with the operator $J_{\le 0}$  is a differential graded  algebra
in which $J_{\le 0}$ is of degree $-1$, that is, $J_{\le 0}(G_k)\subseteq G_{k-1}$.
\end{lemma}
\begin{proof}
We first verify that the product is compatible with the grading.
Since for $[\vec s]\in G_k (k\ge 0)$,
$[\vec s]\shap_{\le 0} {\bf 1}={\bf1}\shap_{\le 0} [\vec s]=[\vec s]\in G_k,$
we only need to prove that, for $\vec s \in \Z _{\le 0}^m \ (m\ge 1)$ with $[\vec s ]\in G_k$ and $\vec t \in \Z _{\le 0}^n\ (n\ge 1)$ with $[\vec t ]\in G_\ell$, there is $[\vec s]\shap_{\le 0} [\vec t]\in \Q G _{k+\ell}$. We will prove by induction on the grade $k$ of the first factor.

For $[\vec s]\in G_1=\{[0]\}$, that is $[\vec s]=[0]$, we have $[0]\shap_{\le 0} [\vec t]=[0, \vec t]\in \Q G_{\ell+1}$  for $[\vec t]\in G_\ell$.

For fixed $p\ge 0$, assume that for any $[\vec s]\in G_k$ with $k\le p$ and any $[\vec t]\in G_\ell $ with $\ell \ge 1$, we have   $[\vec s]\shap_{\le 0} [\vec t]\in\Q G_{k+\ell}.$
Consider $[\vec s]=[s_1, \cdots, s_m]=[s_1,\vec s\,']\in G_{p+1}$. So $m-(s_1+\cdots+s_m)=p+1$, there are two cases.

{\bf Case 1:} If  $s_1=0$,  then  $[\vec s\,']\in G_{p}$, and so
	$$[0, \vec s\,']\shap_{\le 0} [\vec t]=\Big[0, [\vec s\,']\shap_{\le 0} [\vec t]\Big]\in \Q G_{p+\ell+1}.$$

{\bf Case 2:} If $s_1<0$, then  $[s_1+1, \vec s\,']\in G_p$, and by induction hypothesis, so
	\begin{align*}
		[s_1, \vec s\,']\shap_{\le 0} [\vec t]=J_{\le 0}\Big([s_1+1, \vec s\,']\shap_{\le 0} [\vec t]\Big)-[s_1+1, \vec s\,']\shap_{\le 0} J_{\le 0}([\vec t])\in \Q G_{p+\ell+1}.
	\end{align*}

Therefore, $\calhc$ is a graded  algebra.
By definition, $J_{\le 0}$ is of degree $-1$, finishing the proof.
\end{proof}

\subsection{The differential Hopf algebra structure on $\calhd$}
Recall from\,\mcite {GHXZ1} the following family of linear operators $\{\delta_i\ |\ i \in \Z_{\ge 1}\}$ on $\calhd$.
\vsb
{\small
\begin{equation}
	\delta_i: \calhd \longrightarrow \calhd,\quad
	\left\{ \begin{array}{ll}
{\bf1}&\mapsto 0, \\
{ }	[s_1,\cdots,s_k]& \mapsto \left\{\begin{array}{ll}
			\sum_{j=1}^i s_j[s_1,\cdots, s_j+1,\cdots, s_i,\cdots,s_k], & i\le k,  \\
			0,& i>k.
		\end{array}\right.
\end{array} \right .
\mlabel{eq:deltadef}
\end{equation}
}
We adopt the convention that $\delta_i=0$ for $i\le 0$.
Furthermore, the {\bf shifted tensor} $\cks$ for any graded linear operators $A: \calhd \to \calhd$ and $\delta _j: \calhd \to \calhd, j\geq 1,$ is defined by
\vsb
\begin{equation}
\begin{aligned}
A\cks \delta_i: \calhd\otimes \calhd&\longrightarrow \calhd\otimes \calhd, \\
[\vec s]\otimes [\vec t]&\mapsto
	A([\vec s])\otimes \delta_{i-\dep(\vec s)}([\vec t]).
\end{aligned}
	\mlabel{eq:shiftten}
\end{equation}
Also denote
\vsb
\begin{equation}
	p_i:=\delta_i-\delta_{i-1}, \quad  i\in\Z.
	\mlabel{eq:pij}
\end{equation}

\begin{prop}\cite[Thm. 2.5] {GHXZ1}
\mlabel{pp:unique}
There is a unique coproduct $\Deltaa : \calhd\longrightarrow \calhd\otimes \calhd$ characterized by the following properties.
\begin{enumerate}
\item  $\Deltaa({\bf 1})={\bf 1}\otimes {\bf 1};$
\item $\Deltaa([1_k])=\sum\limits_{j=0}^{k}[1_j]\otimes [1_{k-j}];$
\item  The family $\{\delta_i\,|\,i\in \Z_{\ge 1}\}$ is a {\bf shifted coderivation} with respect to $\Deltaa$ in the sense that
\vsb
$$(\id\ \cks \delta_i+\delta_i\otimes \id)\Deltaa=\Deltaa \delta_i, \quad i\geq 1.
$$
Equivalently,
the family $\{p_i\,|\,i\in \Z_{\ge 1}\}$ is a shifted coderivation with respect to $\Deltaa$$:$
\vsb
$$(\id\ \cks p_i+p_i\otimes \id)\Deltaa=\Deltaa p_i, \quad i\geq 1.
$$
\end{enumerate}
Furthermore, with the unit $u_{\ge 1}: \Q\to\calhd , 1\mapsto {\bf 1}$, the counit
$\varepsilon_{\geq 1}: \calhd\longrightarrow \Q, {\bf 1}\mapsto 1, [\vec s]\to 0$
and the weight grading, $(\calhd, \shap_{\ge 1}, u_{\ge 1}, \Deltaa, \varepsilon_{\geq 1})$  is a connected Hopf algebra.
\end{prop}

The notion of differential Hopf algebras, which integrates the structures of derivations, algebras and coalgebras, has been defined and discussed in various contexts of mathematics and mathematical physics \mcite{Ara, AY,Br, Dra, HM, Schu}. Often the derivation $d$ is assumed to be nilpotent: $d^2=0$ and there is no interrelation on the derivation and the antipode.
We will use the following variation.

\begin{defn} A {\bf differential Hopf algebra} is a Hopf algebra $(H, m, u, \Delta, \vep,S)$ equipped with a linear operator $d$ such that
\begin{enumerate}
\item $d$ is a derivation with respect to $m$: $d(xy)=d(x)y+xd(y), x,y\in H$,
\item $d$ is a {\bf coderivation} with respect to $\Delta$:
\vsa
\begin{equation}
	\Delta d= (\id \ot d+d \ot \id)\Delta,
\mlabel{eq:coder}
\end{equation}
\item $d$ commutes with the antipode: $d\, S=Sd$.
\end{enumerate}
A bialgebra $(H, m, u,\Delta,\vep)$ equipped with a derivation $d$ which is also a coderivation is called a {\bf differential bialgebra}.
\mlabel{d:diffhopf}	
\end{defn}
We have the following coalgebra counterpart of the fact that $d(1)=0$ for a derivation $d$.
\begin{lemma}\label{lem:derandcoder}
Let $(H, \Delta, \vep)$ be a coassociative coalgebra. If $d: H\to H$ is a coderivation with respect to $\Delta$, then $\vep d=0$.
\end{lemma}
\begin{proof}
Applying $(\vep\otimes \id)$ to the coderivation property
$\Delta d=(d\otimes \id+\id \otimes d)\Delta$ gives
\vsb
$$(\vep\otimes \id)\Delta d=(\vep\otimes \id)(d\otimes \id+\id \otimes d)\Delta.
$$
The left-hand side is $d$ due to the counit property $(\vep\otimes \id)\Delta=\id$.
The right-hand side is
\vsa
\begin{equation*}
\begin{aligned}
(\vep\otimes \id)(d\otimes \id+\id \otimes d)\Delta&=(\vep d\otimes\id)\Delta+(\vep \otimes d)\Delta
\delete{
&=(\vep d\otimes\id)\Delta(x)+\sum \vep(x_{i}')d(x_{i}'')\\
&=(\vep d\otimes\id)\Delta(x)+d(\sum \vep(x_{i}')x_{i}'')\\
&=(\vep d\otimes\id)\Delta(x)+d(\vep\otimes\id)\Delta(x)\\
}
=(\vep d\otimes\id)\Delta+d.
\end{aligned}
\end{equation*}
Thus
$d=(\vep d\otimes\id)\Delta+d,
$
giving
\vsb
$$(\vep d\otimes\id)\Delta=0.$$
Applying $(\id\otimes \vep)$ and using the counit property  $(\id\otimes \vep)\Delta=\id$, we obtain
\vsa
$$0=(\id\ot \vep) (\vep d\otimes\id)\Delta= (\vep d\otimes \vep)\Delta=(\vep d\otimes\id)(\id\otimes \vep)\Delta=\vep d\otimes \id.
 $$
This means that, for any $x\in H$ and a fixed basis element $y\in H$, we have
\vsa
$$ 0=(\vep d\otimes \id)(x\otimes y) = (\vep d)(x)\, y,$$
noting that $(\vep d)(x)$ is a scalar. This proves $\vep d=0$.
\end{proof}

Recall that a {\bf graded bialgebra} is a bialgebra $(H, m,u,\Delta,\vep)$ with a grading
$H=\bigoplus\limits_{k=0}^\infty H_k$ such that $(H, m,u)$ is a graded algebra and
{\small
	$$\Delta(H_k)\subseteq \bigoplus_{i+j=k} H_i\otimes H_j, k\geq 0.$$
}It is called {\bf connected} if
\vsb
{\small
$$H_0={\bf k} 1, \quad \ker \vep=\bigoplus_{k\geq 1} H_k.
\vsa$$
}
As is well-known~\mcite{Man}, a connected graded bialgebra is a Hopf algebra.
When $H$ is connected, the coproduct of an element $x$ of degree $n$ is given as:
\vsa
\[
\Delta (x) = x \otimes 1 + 1 \otimes x + \sum_i x'_i \otimes x''_i,
\vsa
\]
where the degrees of $x'_i$ and $x''_i$ are strictly between 0 and $n$.
Then the antipode $S$ is defined by the recursion
\vsb
\begin{equation}
	\mlabel{eq:connantipode}
S(1)=1, \quad S(x) = -x - \sum_i S(x'_i) x''_i, x\in H_k, k\geq 1.
\vsb
\end{equation}

Enriching to the differential context, we obtain the following result.

\begin{theorem}
A connected differential bialgebra $(H, m, \Delta, u, \vep, d)$ is a differential Hopf algebra.
\mlabel{thm:autoSd=dS}
\end{theorem}

\begin{proof}
As cited above, the connected bialgebra $H$ is a Hopf algebra. We just need to verify $Sd=dS$ for which we prove $Sd(x)=dS(x)$ for $x\in H_n$ by induction on $n\geq 0$.

The case when $x\in H_0 (=\mathbf{k} 1)$ follows directly from $d(1)=0$.

Fix $n\geq 0$ and assume that $Sd(x) = dS(x)$ hold for all elements $x$ of grade less than $n$. Let $x \in H_n$.
Applying $d$ to Eq.\,\meqref{eq:connantipode}, the Leibniz rule and  the induction hypothesis gives
\vsa
\begin{equation}
d(S(x))
= -d(x) - \sum_i \Big( d(S(x'_i)) x''_i + S(x'_i) d(x''_i) \Big)
= -d(x) - \sum_i \Big( S(d(x'_i)) x''_i + S(x'_i) d(x''_i) \Big). \label{eq:dS}
\vsa
\end{equation}

Next we compute $S(d(x))$. Since $d$ is a coderivation and $d(1)=0$, we have
\vsb
{\small
\begin{equation}\label{eq:Deltad}
\begin{split}
\Delta d(x) =& (d \otimes \text{id} + \text{id} \otimes d) \bigg( x \otimes 1 + 1 \otimes x + \sum_i x'_i \otimes x''_i \bigg)
\\
=& d(x) \otimes 1 + d(1) \otimes x + 1 \otimes d(x) + x \otimes d(1) + \sum_i \left( d(x'_i) \otimes x''_i + x'_i \otimes d(x''_i) \right)\\
=& d(x) \otimes 1 + 1 \otimes d(x) + \sum_i \bigg( d(x'_i) \otimes x''_i + x'_i \otimes d(x''_i) \bigg).
\end{split}
\vsa
\end{equation}
}
Applying the defining identity $m (S \otimes \id) \Delta (x) =
u \vep(x)
$ of $S$ to $d(x)$ and using Lemma\,\ref{lem:derandcoder} yield
\vsb
\[
m (S \otimes \text{id}) \Delta d(x) = u \vep(d(x))=0.
\vsa
\]
Expanding by Eq.\,(\ref{eq:Deltad}) gives
\vsa
\[
0=m (S \otimes \text{id}) \Delta d(x) = S(d(x)) \cdot 1 + S(1) \cdot d(x) + \sum_i \Big( S(d(x'_i)) x''_i + S(x'_i) d(x''_i) \Big).
\]
Since $S(1)=1$, this becomes
\vsa
\[
S(d(x)) + d(x) + \sum_i \Big( S(d(x'_i)) x''_i + S(x'_i) d(x''_i) \Big) = 0.
\]
Therefore,
\vsa
\begin{equation*}
\label{eq:Sd}
S(d(x)) = -d(x) - \sum_i \Big( S(d(x'_i)) x''_i + S(x'_i) d(x''_i) \Big).
\vsa
\end{equation*}
This agrees with Eq.\,(\ref{eq:dS}),
which completes the induction.
\end{proof}

To prove the coderivation property of  $J_{\ge 1}$, we first give a preparational lemma.

\begin{lemma}\mlabel{lem:deltaiandJ}  For $ [s_1,\cdots, s_k]\in \calh^{+}_{\Z\ge 1}$ and $1\le i\le k$,
$$
(J_{\ge 1} \delta_i-\delta_i J_{\ge 1})([s_1,\cdots\hskip-1mm, s_k])=[s_1,\cdots\hskip-1mm, s_k].
$$
\end{lemma}

\begin{proof}
By the definition of $\delta_i$ and $J_{\ge 1}$, we verify the desired identity in two cases.

\smallskip
\noindent
{\bf Case 1:} If $s_1=1$, then
\vsc
 \begin{equation*}
\begin{split}
(J_{\ge 1} \delta_i-\delta_i J_{\ge 1})([1, s_2, \cdots\hskip-1mm, s_k])&=J_{\ge 1}\Big([2, s_2,\cdots\hskip-1mm,s_k]+\sum_{j=2}^is_j[1,s_2,\cdots\hskip-1mm,s_j+1,\cdots]\Big)\\
&=[1,s_2,\cdots\hskip-1mm,s_k].
\end{split}
\vsb\end{equation*}

\noindent
{\bf Case 2:} If $s_1>1$, then
\vsc
\begin{equation*}
\begin{split}
&(J_{\ge 1} \delta_i-\delta_i J_{\ge 1})([s_1,\cdots\hskip-1mm, s_k])=J_{\ge 1}\bigg(\sum_{j=1}^is_j[\cdots\hskip-1mm,s_j+1,\cdots\hskip-1mm, s_k]\bigg)-\delta_i([s_1-1, s_2,\cdots\hskip-1mm, s_k])\\
=&s_1[s_1,\cdots\hskip-1mm,s_k]+\sum_{j=2}^is_j[s_1-1,\cdots\hskip-1mm,s_j+1,\cdots]-(s_1-1)[s_1,\cdots\hskip-1mm,s_k]\\
 &-\Big(\sum_{j=2}^i s_j[s_1-1,\cdots\hskip-1mm,s_j+1,\cdots]\Big)\\
=&[s_1,\cdots\hskip-1mm, s_k]. \qedhere
\end{split}
\vsb
\end{equation*}
\end{proof}

In terms of $p_i$ in Eq.~\meqref{eq:pij},  we have the following direct consequence.

\begin{coro}\mlabel{coro:deltaiandJ}
For $ [s_1,\cdots\hskip-1mm, s_k]\in \calh^{+}_{\Z\ge 1}$,
\begin{enumerate}
 \item $i=1$,
\vsb
  $$(J_{\ge 1} p_1-p_1 J_{\ge 1})([s_1,\cdots\hskip-1mm, s_k])=[s_1,\cdots\hskip-1mm, s_k].
  $$
  \item $2\le i\le k$,
\vsb
  $$
J_{\ge 1} p_i([s_1,\cdots\hskip-1mm, s_k])=p_i J_{\ge 1}([s_1,\cdots\hskip-1mm, s_k]).
$$
\end{enumerate}
\end{coro}

We next prove the coderivation property of $J_{\ge 1}$, beginning with a special case.

\begin{lemma}\mlabel{lem:1vecs}
For $[1, \vec s]\in \calh^+_{\Z_{\ge 1}}$, we have
\vsa
$$(\id\otimes J_{\ge 1}+J_{\ge 1}\otimes \id)\Deltaa([1, \vec s])=\Deltaa J_{\ge 1}([1, \vec s]).
\vsb$$
\end{lemma}

\begin{proof}
By definition of $J_{\ge 1}$, we have $\Deltaa J_{\ge 1}([1, \vec s])=0$. For the left-hand side, let $\vec s=(s_2,\cdots\hskip-1mm, s_k)$, $k\geq 1$. Here when $k=1$, $[1,\vec{s}]=[1]$ by convention, in which case $(\id\otimes J_{\ge 1}+J_{\ge 1}\otimes \id)\Deltaa([1])=0$. Thus we only need to consider the case when $k>1$, for which we prove $(\id\otimes J_{\ge 1}+J_{\ge 1}\otimes \id)\Deltaa([1, \vec s])=0$ by induction on the weight $n:=s_2+\cdots+s_k\ge k-1$ of $\vec s$. For the initial step $n=k-1$, that is $s_2=\cdots=s_k=1$, we have
\vsb
\begin{equation*}
(\id\otimes J_{\ge 1}+J_{\ge 1}\otimes \id)\Deltaa([1_k])=(\id\otimes J_{\ge 1}+J_{\ge 1}\otimes \id)\Big(\sum_{j=0}^k[1_j]\otimes [1_{k-j}]\Big)=0.
\vsa
\end{equation*}
Fix $\ell\geq k-1$, and assume that $(\id\otimes J_{\ge 1}+J_{\ge 1}\otimes \id)\Deltaa ([1, s_2,\cdots\hskip-1mm,s_k])=0$ for $n=s_2+\cdots+s_k=\ell$. Then for $n=\ell+1$, there is some $2\le i\le k$ such that $s_i>1$.  Then by definition of $p_i$, we have
\begin{align*}
 &(\id\otimes J_{\ge 1}+J_{\ge 1}\otimes \id)\Deltaa ([1, s_2,\cdots\hskip-1mm,s_k])=\frac{(\id\otimes J_{\ge 1}+J_{\ge 1}\otimes \id)\Deltaa}{s_i-1}p_i([1,\cdots\hskip-1mm,s_i-1,\cdots])\\
\overset{\textcircled{1}}{=}&\frac{(\id\otimes J_{\ge 1}+J_{\ge 1}\otimes \id)\big(\id\cks p_i+p_i\otimes \id\big)}{s_i-1}\Deltaa([1,\cdots\hskip-1mm,s_i-1,\cdots])\\
=&\frac{\id\cks J_{\ge 1}p_i+p_i\otimes J_{\ge 1}+J_{\ge 1}\cks p_i+J_{\ge 1}p_i\otimes\id}{s_i-1}\Deltaa([1,\cdots\hskip-1mm,s_i-1,\cdots])\\
\overset{\textcircled{2}}{=}&\frac{\id\cks p_iJ_{\ge 1}+p_i\otimes J_{\ge 1}+J_{\ge 1}\cks p_i+p_iJ_{\ge 1}\otimes\id}{s_i-1}\Deltaa([1,\cdots\hskip-1mm,s_i-1,\cdots])\\
=&\frac{\big(\id\cks p_i+p_i\otimes \id\big)(\id\otimes J_{\ge 1}+J_{\ge 1}\otimes \id)}{s_i-1}\Deltaa([1,\cdots\hskip-1mm,s_i-1,\cdots])\\
\overset{\textcircled{3}}{=}&\frac{\id\cks p_i+p_i\otimes \id}{s_i-1}(0)=0.
\end{align*}
Here $\textcircled{1}$ follows from Proposition \mref{pp:unique}; $\textcircled{2}$ follows from Corollary \mref{coro:deltaiandJ}; and $\textcircled{3}$ follows from the induction hypothesis and the definition of $J_{\ge 1}$. Hence $(\id\otimes J_{\ge 1}+J_{\ge 1}\otimes \id)\Deltaa([1, \vec s])=0$, completing the induction.
\end{proof}

For general bases elements, we have
\begin{prop}\mlabel{prop:Jge1commuteDeltage1}
The operator $J_{\geq 1}$ is a coderivation on the coalgebra $(\calhd,\Deltaa):$
$$(\id\otimes J_{\ge 1}+J_{\ge 1}\otimes \id)\Deltaa=\Deltaa J_{\ge 1}.
$$
\end{prop}
\begin{proof}
First it is evident that $(\id\otimes J_{\ge 1}+J_{\ge 1}\otimes \id)\Deltaa({\bf1})=0=\Deltaa J_{\ge 1}({\bf1})$.

Next, for $[s_1, \vec s]\in \calh^{+}_{\Z_{\ge 1}}$, we verify the identity by an induction on $s_1\geq 1$, with the base case of $s_1=1$ already accomplished in Lemma \mref{lem:1vecs}.
Assume that the conclusion holds for $s_1=\ell$ for a given $\ell \ge 1$. Then for $s_1=\ell+1$, we have
\vsb
\begin{align*}
&(\id\otimes J_{\ge 1}+J_{\ge 1}\otimes \id)\Deltaa([\ell+1, \vec s])=\frac{(\id\otimes J_{\ge 1}+J_{\ge 1}\otimes \id)\Deltaa}{\ell}\delta_1([\ell, \vec s])\\
\overset{\textcircled{1}}{=}&\frac{(\id\otimes J_{\ge 1}+J_{\ge 1}\otimes \id)(\id\cks\delta_1+\delta_1\otimes\id)}{\ell}\Deltaa([\ell, \vec s])\\
=&\frac{\id\cks J_{\ge 1}\delta_1+\delta_1\otimes J_{\ge 1}+J_{\ge 1}\cks\delta_1+J_{\ge 1}\delta_1\otimes\id}{\ell}\Deltaa([\ell, \vec s])\\
\overset{\textcircled{2}}{=}&\frac{{\bf1}\otimes [\ell, \vec s]}{\ell}+\frac{\id\cks \delta_1J_{\ge 1}}{\ell}({\bf1}\otimes [\ell, \vec s])+\frac{\delta_1\otimes J_{\ge 1}+J_{\ge 1}\cks\delta_1}{\ell}\Deltaa([\ell, \vec s])\\
&+\frac{J_{\ge 1}\delta_1\otimes\id}{\ell}\bigg(\Deltaa([\ell, \vec s])-{\bf 1}\otimes [\ell, \vec s]\bigg)\\
\overset{\textcircled{3}}{=}&\frac{{\bf1}\otimes [\ell, \vec s]}{\ell}+\frac{\id\cks \delta_1J_{\ge 1}}{\ell}({\bf1}\otimes [\ell, \vec s])+\frac{\delta_1\otimes J_{\ge 1}+J_{\ge 1}\cks\delta_1}{\ell}\Deltaa([\ell, \vec s])\\
&+\frac{\delta_1J_{\ge 1}\otimes \id}{\ell}\Deltaa([\ell, \vec s])+\frac{\Deltaa([\ell, \vec s])-{\bf 1}\otimes [\ell, \vec s]}{\ell}\\
=&\frac{\Deltaa([\ell, \vec s])}{\ell}+\frac{(\id\cks \delta_1+\delta_1\otimes \id)(\id\otimes J_{\ge 1}+J_{\ge 1}\otimes\id)}{\ell}\Deltaa([\ell, \vec s])\\
\overset{\textcircled{4}}{=}&\frac{\Deltaa([\ell, \vec s])}{\ell}+\frac{(\ell-1)}{\ell}\Deltaa([\ell, \vec s])\\
=&\Deltaa([\ell, \vec s])=\Deltaa J_{\ge 1}([\ell+1, \vec s]).
\end{align*}
Here $\textcircled{1}$ follows from Proposition \mref{pp:unique} for $i=1$; $\textcircled{2}$ follows from the facts that
\vsa
$$(\id\cks \delta_1)\Deltaa([\ell, \vec s])={\bf1}\otimes [\ell, \vec s],\, (\delta_1\otimes \id)\Deltaa([\ell, \vec s])=(\delta_1\otimes \id)\Big(\Deltaa([\ell, \vec s])-{\bf1}\otimes [\ell, \vec s]\Big)$$ and Lemma \mref{lem:deltaiandJ} for $i=1$; $\textcircled{3}$ follows from Lemma \mref{lem:deltaiandJ} for $i=1$ and the fact that $(\delta_1\otimes\id)({\bf1}\otimes [\ell, \vec s])=0$; and  $\textcircled{4}$ follows from the induction hypothesis and Proposition \mref{pp:unique} for $i=1$. This completes the induction.
\end{proof}

\begin{theorem}\mlabel{thm:ge1diffHopf}
Let $S_{\ge 1}$ be the antipode of the connected Hopf algebra $\calhd$. The septuple $(\calhd, \shap_{\ge 1}, u_{\ge 1}, \Deltaa, \varepsilon_{\geq 1}, S_{\ge 1}, J_{\ge 1})$ is a differential Hopf algebra.
\end{theorem}
\begin{proof}
By \mcite{GHXZ1}, we have that $(\calhd, \shap_{\ge 1}, u_{\ge 1}, \Deltaa, \varepsilon_{\geq 1}, S_{\ge 1})$ is a connected Hopf algebra. Then the conclusion follows from Lemma \mref{lem:Jge1der}, Proposition \mref{prop:Jge1commuteDeltage1} and Theorem~\mref{thm:autoSd=dS}.
\end{proof}

\subsection{The differential Hopf algebra structure on $\calhc$}
We now turn to study the Hopf algebra structure on $\calhc$. By Proposition \mref{lem:le0differentialA} and Lemma \mref {lem:suba}, it is a differential algebra, as well as a graded algebra. A simple induction on $-s_1\ge 0$ establishes the following result.
\begin{prop}
\mlabel{prop:uniquedeltale0}
There  is  a  unique  linear  map  $\Deltab: \calhc\longrightarrow\calhc\otimes \calhc$  such  that
\begin{enumerate}
\item $\Deltab({\bf 1})={\bf 1}\otimes {\bf 1}$;
\item $\Deltab([0])={\bf 1}\otimes [0]+[0]\otimes {\bf 1}$;
\item $\Deltab J_{\le 0}=(\id \otimes J_{\le 0}+J_{\le 0}\otimes \id)\Deltab$;
\item  $\Deltab([0, \vec s])=\Big(\Deltab ([0])\Big)\shap_{\le 0} \Big(\Deltab ([\vec s])\Big)$  for $\vec s\in \Z_{\le 0}^\ell \ (\ell \ge 1)$.
\end{enumerate}
\end{prop}

Here are some examples to illustrate how the properties are applied.
\begin{equation*}
\begin{split}
\Deltab([-1])&=\Deltab J_{\le 0}([0])=(\id \otimes J_{\le 0}+J_{\le 0}\otimes \id)\Deltab([0])={\bf 1}\otimes [-1]+[-1]\otimes {\bf 1},\\
\Deltab([0,0])&=\bigg(\Deltab([0])\bigg)\shap_{\le 0} \bigg(\Deltab([0])\bigg)={\bf 1}\otimes [0,0]+2[0]\otimes [0]+[0,0]\otimes {\bf 1},\\
 \Deltab([-1, 0])&=\Deltab J_{\le 0}([0,0])=(\id \otimes J_{\le 0}+J_{\le 0}\otimes \id)\Deltab([0,0])\\
 &={\bf 1}\otimes [-1, 0]+2[0]\otimes [-1]+2[-1]\otimes [0]+[-1, 0]\otimes {\bf 1}.
\end{split}
\end{equation*}

\begin{lemma}
\mlabel{lem:Deltale0coassociative}
The linear map $\Deltab$ in Proposition \mref {prop:uniquedeltale0} is  coassociative and cocommutative.
\end{lemma}
\begin{proof} Obviously, $(\id \otimes\Deltab)\Deltab ({\bf 1})=(\Deltab\otimes \id)\Deltab ({\bf 1})$. So we only need to establish the coassociativity for basis element $[\vec s]$ with $\vec s\in \Z_{\le 0}^m$, for which we utilize induction  on the depth $m\geq 1$.

For $m=1$, denote
$$Z_n:=\{s\in \Z_{\le 0}\ |\  s=-n\},
$$
$$T:=\{s\in\Z_{\le 0}\ |\  (\id \otimes\Deltab)\Deltab ([s])=(\Deltab\otimes \id)\Deltab ([s])\}.
$$
The definition of $\Deltab$ gives  $Z_0\subset T$. Assume $Z_n\subset T$ for a given $n\geq 0$. Then for $s\in Z_{n+1}$, we have
\begin{align*}
&(\id \otimes\Deltab)\Deltab ([s])=(\id \otimes\Deltab)\Deltab (J_{\le 0}([s+1]))\\
=&(\id \otimes\Deltab)( \id \otimes J_{\le 0}+J_{\le 0}\otimes \id)\Deltab ([s+1])\\
=&(\id \otimes \Deltab J_{\le 0}+J_{\le 0}\otimes \Deltab)\Deltab([s+1])\\
=&(\id \otimes \id \otimes J_{\le 0}+\id \otimes J_{\le 0}\otimes \id+J_{\le 0}\otimes \id \otimes \id)(\id \otimes\Deltab)\Deltab([s+1])\\
\overset{\textcircled{1}}{=}&(\id \otimes \id \otimes J_{\le 0}+\id \otimes J_{\le 0}\otimes \id+J_{\le 0}\otimes \id \otimes \id)(\Deltab\otimes \id)\Deltab ([s+1])\\
=&(\Deltab\otimes \id)( \id \otimes J_{\le 0}+J_{\le 0}\otimes \id)\Deltab ([s+1])\\
=&(\Deltab\otimes \id)\Deltab \bigg(J_{\le 0}([s+1])\bigg)=(\Deltab\otimes \id)\Deltab ([s]).
\end{align*}
Here equality $\textcircled{1}$ follows from the induction hypothesis and the other equalities use the coderivation property of $J_{\le 0}$. Thus, $Z_{n+1}\subset T$. Hence  $T=\Z_{\le 0}$, and the coassociativity holds for $m=1$.

Assume that the conclusion holds for $m=k$ with given $k\ge 1$.   Then consider $m=k+1$. Let
\vsa
$$Z_n:=\{(s_1, \vec s) \in \Z_{\le 0}^{k+1}\ |\  s_1=-n\},
$$
$$T:=\{(s_1, \vec s) \in \Z_{\le 0}^{k+1}\ |\  (\id \otimes\Deltab)\Deltab ([s_1, \vec s])=(\Deltab\otimes \id)\Deltab ([s_1, \vec s])\}.
$$
Then for $(0, \vec s)\in Z_0$, by the definition of $\Deltab$ in Proposition \mref{prop:uniquedeltale0}, we have
\vsb
\begin{align*}
(\id \otimes\Deltab)\Deltab ([0, \vec s])&=\bigg((\id \otimes\Deltab)\Deltab ([0])\bigg)\shap_{\le 0} \bigg((\id \otimes\Deltab)\Deltab ([\vec s])\bigg)\\
&=\bigg((\Deltab\otimes \id)\Deltab ([0])\bigg)\shap_{\le 0} \bigg((\Deltab\otimes \id)\Deltab ([\vec s])\bigg)\\
&=(\Deltab\otimes \id)\Deltab ([0, \vec s]).
\end{align*}
Hence $Z_0\subset T$. Next assume that $Z_n\subset T$ with $n\ge 0$. Then for $(s_1, \vec s)\in Z_{n+1}$, we have
\vsb
\begin{align*}
&(\id \otimes\Deltab)\Deltab ([s_1, \vec s])=(\id \otimes\Deltab)\Deltab\bigg(J_{\le 0}([s_1+1, \vec s])\bigg)\\
=&(\id \otimes\Deltab)( \id \otimes J_{\le 0}+J_{\le 0}\otimes \id)\Deltab([s_1+1, \vec s])\\
=&(\id \otimes \id \otimes J_{\le 0}+ \id \otimes J_{\le 0}\otimes \id+J_{\le 0}\otimes \id \otimes \id)(\id \otimes\Deltab)\Deltab([s_1+1, \vec s])\\
\overset{\textcircled{1}}{=}&(\id \otimes \id \otimes J_{\le 0}+ \id \otimes J_{\le 0}\otimes \id+J_{\le 0}\otimes \id \otimes \id)( \Deltab \otimes \id)\Deltab([s_1+1, \vec s])\\
=&( \Deltab \otimes \id)\Deltab\bigg(J_{\le 0}([s_1+1, \vec s])\bigg)\\
=&( \Deltab \otimes \id)\Deltab([s_1, \vec s]).
\end{align*}
Here equality $\textcircled{1}$ follows from the induction hypothesis and the other equalities use the coderivation property of $J_{\le 0}$. Thus $Z_{n+1}\subset T$, and hence $T=\Z_{\le 0}^{k+1}$. Therefore,  $\Deltab$  is coassociative.

We next show that $\Deltab$ is cocommutative by induction on the grading $k\ge 0$ in Eq. (\mref{eq:grade<1}). Let $\tau: \calhc\otimes \calhc\to \calhc\otimes \calhc, \tau(a\otimes b)=b\otimes a$ be the flip map. Since $G_0=\{{\bf1}\}$, obviously $\tau\Deltab({\bf1})={\bf1}\otimes {\bf1}=\Deltab({\bf1})$. Now it is enough to prove that $\Deltab$ is cocommutative on basis $G_k$ for $k\ge 1$.
For the initial step $k=1$, $G_1=\{[0]\}$, by the definition of $\Deltab([0])$, we get
\vsa
$$\tau\Deltab([0])={\bf1}\otimes [0]+[0]\otimes {\bf1}=\Deltab([0]).
$$
Then for $k=\ell\ge 1$, assume that $\tau\Deltab([\vec s])=\Deltab([\vec s])$ for $[\vec s]\in G_{\ell}$. Next, for $[s_1, \vec s]\in G_{\ell+1}$, we prove $\tau\Deltab([s_1, \vec s])=\Deltab([s_1, \vec s])$ by breaking the argument into two cases.\\
{\bf Case 1:} If $s_1=0$, then $[\vec s]\in G_\ell$. Write
\vsb
$$\Deltab([\vec s])=\sum a^{\vec s}_{\vec u, \vec v}[\vec u]\otimes [\vec v].
$$
By the induction hypothesis,
\vsb
\begin{equation}
\Deltab([\vec s])=\sum a^{\vec s}_{\vec u, \vec v}[\vec u]\otimes [\vec v]=\sum a^{\vec s}_{\vec u, \vec v}[\vec v]\otimes [\vec u]=\tau\Deltab([\vec s]).\mlabel{eq:cocommu}
\end{equation}
Thus
\vsb
\begin{align*}
\tau\Deltab([0, \vec s])&=\tau\bigg(\Big(\Deltab([0])\Big)\shap_{\le 0}\Big(\Deltab([\vec s])\Big)\bigg)=\tau\bigg(\Big({\bf1}\otimes [0]+[0]\otimes {\bf 1}\Big)\shap_{\le 0} \Big(\sum a^{\vec s}_{\vec u, \vec v}[\vec u]\otimes [\vec v]\Big)\bigg)\\
&=\tau\bigg(\sum a^{\vec s}_{\vec u, \vec v}[\vec u]\otimes [0, \vec v]+\sum a^{\vec s}_{\vec u, \vec v}[0,\vec u]\otimes [\vec v]\bigg)\\
&=\sum a^{\vec s}_{\vec u, \vec v}[0,\vec v]\otimes [\vec u]+\sum a^{\vec s}_{\vec u, \vec v}[\vec v]\otimes [0, \vec u]\\
&=\Big({\bf 1}\otimes [0]+[0]\otimes {\bf 1}\Big)\shap_{\le 0} \bigg(\sum a^{\vec s}_{\vec u, \vec v}[\vec v]\otimes [\vec u]\bigg)\\
&\overset{(\ref{eq:cocommu})}{=}\Big({\bf 1}\otimes [0]+[0]\otimes {\bf 1}\Big)\shap_{\le 0} \bigg(\sum a^{\vec s}_{\vec u, \vec v}[\vec u]\otimes [\vec v]\bigg)\\
&=\Big(\Deltab([0])\Big)\shap_{\le 0}\Big(\Deltab\big([\vec s]\big)\Big)=\Deltab([0, \vec s]).
\end{align*}
{\bf Case 2:} If $s_1<0$, then $[s_1+1, \vec s]\in G_{\ell}$. By the induction hypothesis,
\vsa
$$\tau\Deltab([s_1+1, \vec s])=\Deltab([s_1+1, \vec s]).
$$
Hence,
\vsb
\begin{equation*}
\begin{split}
\tau\Deltab([s_1, \vec s])&=\tau\Deltab J_{\le 0}([s_1+1, \vec s])=\tau (\id\otimes J_{\le 0}+J_{\le 0}\otimes \id)\Deltab([s_1+1, \vec s])\\
&\overset{\textcircled{1}}{=}(\id\otimes J_{\le 0}+J_{\le 0}\otimes \id)\tau\Deltab([s_1+1, \vec s])\\
&\overset{\textcircled{2}}{=}(\id\otimes J_{\le 0}+J_{\le 0}\otimes \id)\Deltab([s_1+1, \vec s])\\
&=\Deltab J_{\le 0}([s_1+1,\vec s])=\Deltab([s_1, \vec s]).
\end{split}
\end{equation*}
Here $\textcircled{1}$ is from $\tau (\id\otimes J_{\le 0}+J_{\le 0}\otimes \id)=(\id\otimes J_{\le 0}+J_{\le 0}\otimes \id)\tau$ by a direct calculation, and $\textcircled{2}$ is from the induction hypothesis $\tau\Deltab([s_1+1, \vec s])=\Deltab([s_1+1, \vec s])$. In summary, $\Deltab$ is cocommutative.
\end{proof}

By a direct calculation, we obtain
\begin{lemma}
 \mlabel{lem:1tensorJ+Jtensor1}
 The linear operator
 $$(\id \otimes J_{\le 0}+J_{\le 0}\otimes \id): \calhc\otimes \calhc\longrightarrow \calhc\otimes \calhc$$
 is a derivation with respect to the shuffle product $\shap_{\le 0}$ on $\calhc\otimes \calhc$ defined by
 $$(a\otimes b)\shap_{\le 0} (c\otimes d)=(a\shap_{\le 0} b)\otimes(c\shap_{\le 0} d), \quad a, b, c, d\in \calhc.$$
\end{lemma}

\begin{lemma}
\mlabel{lem:D<0homomorphism}
The linear map $\Deltab$ in Proposition \mref {prop:uniquedeltale0}  is  a  homomorphism  with respect to the product $\shap_{\le 0}$.
\end{lemma}

\begin{proof}
  It suffices to prove it for $\vec s\in \Z_{\le 0}^m(m\ge 1)$ and $\vec t\in \Z_{\le 0}^p(p\ge 1)$:
\vsb
\begin{equation} \mlabel{eq:deltahom}
\Deltab\big([\vec s]\shap_{\le 0} [\vec t]\big)=\Big(\Deltab([\vec s])\Big)\shap_{\le 0}\Big(\Deltab([\vec t])\Big),
\end{equation}
for which we will apply the induction on $m+p\geq 2$. For the initial step $m+p=2$, that is $m=p=1$. So $[\vec s]=[s]$ and $[\vec t]=[t]$. For $n\ge 0$, denote
\vsb
\begin{eqnarray*}
	&Z_n:=\{s\in \Z_{\le 0}\ |\ s=-n\},& \\
&T:=\{s\in \Z_{\le 0}\ |\ \Deltab([s]\shap_{\le 0} [t])=\Big(\Deltab([s])\Big)\shap_{\le 0}\Big(\Deltab([t])\Big)\}.&
\end{eqnarray*}
Then Eq.\,\meqref{eq:deltahom} in this case means that $Z_n\subset T$ for all $n\geq 0$, for which we apply induction on $-s\geq 0$.
By  definition, $\Deltab([0]\shap_{\le 0} [t])=\Big(\Deltab([0])\Big)\shap_{\le 0} \Big(\Deltab([t])\Big)$, and so $Z_0\subset T$. Assume  $Z_n\subset T$ for $-s=n\ge 0$.  Then for $s\in Z_{n+1}$,  we get
\small{
\begin{align*}
&\Deltab([s]\shap_{\le 0} [t])\\
=&\Deltab\Big(\Big(J_{\le 0}([s+1])\Big)\shap_{\le 0} [t]\Big)=\Deltab J_{\le 0}\Big([s+1]\shap_{\le 0} [t]\Big)-\Deltab\Big([s+1]\shap_{\le 0} [t-1]\Big)\\
\overset{\textcircled{1}}{=}&( \id \otimes J_{\le 0}+J_{\le 0}\otimes \id)\Deltab\Big([s+1]\shap_{\le 0} [t]\Big)-\Big(\Deltab ([s+1])\Big)\shap_{\le 0} \Big(\Deltab([t-1])\Big)\\
=&( \id \otimes J_{\le 0}+J_{\le 0}\otimes \id)\Big(\Big(\Deltab([s+1])\Big)\shap_{\le 0} \Big(\Deltab([t])\Big) \Big)-\Big(\Deltab ([s+1])\Big)\shap_{\le 0} \Big(\Deltab([t-1])\Big)\\
\overset{\textcircled{2}}{=}&\Big( ( \id \otimes J_{\le 0}+J_{\le 0}\otimes \id)\Deltab([s+1])\Big)\shap_{\le 0} \Big(\Deltab([t])\Big)+\Big(\Deltab([s+1])\Big)\shap_{\le 0} \Big( ( \id \otimes J_{\le 0}+J_{\le 0}\otimes \id)\Deltab([t])\Big)\\
&-\Big(\Deltab ([s+1])\Big)\shap_{\le 0} \Big(\Deltab([t-1])\Big)\\
=&\Big(\Deltab J_{\le 0}([s+1])\Big)\shap_{\le 0} \Big(\Deltab([t])\Big)+\Big(\Deltab([s+1])\Big)\shap_{\le 0} \Big(\Deltab J_{\le 0}([t])\Big)-\Big(\Deltab ([s+1])\Big)\shap_{\le 0} \Big(\Deltab([t-1])\Big)\\
=&\Big(\Deltab([s])\Big)\shap_{\le 0} \Big(\Deltab([t])\Big).
\end{align*}
}
Here $\textcircled{1}$ followed by by Proposition \mref{prop:uniquedeltale0} and the induction hypothesis $\Deltab([s+1]\shap_{\le 0} [t-1])=\Big(\Deltab([s+1])\Big)\shap_{\le 0}\Big(\Deltab([t-1])\Big)$ because of $s+1\in Z_n$, and $\textcircled{2}$ followed by Proposition \mref{prop:uniquedeltale0} and Lemma \mref{lem:1tensorJ+Jtensor1}. Hence we have proved
$Z_{n+1}\subset T$.
This completes the inductive proof of Eq.~\meqref{eq:deltahom} for the case of $m+p=2$.

Assume that Eq.\,\meqref{eq:deltahom} holds when the sum of depths  $m+p$ is $k$ for a given $k\geq 2$. Then for $(s_1, \vec s\,')\in\Z _{\le 0}^m$ and $\vec t \in \Z _{\le 0}^p$ with $m+p=k+1$, denote
\vsb
\begin{eqnarray*}
	&Z_n:=\{s_1\in \Z_{\le 0}\ |\ s_1=-n\},& \\
&T:=\Big\{s_1\in \Z_{\le 0}\ |\ \Deltab([s_1, \vec s\,']\shap_{\le 0} [\vec t])=\big(\Deltab([s_1, \vec s\,'])\big)\shap_{\le 0} \big(\Deltab([\vec t])\big)\Big\}.&
\end{eqnarray*}
As in the case of $m+p=2$, Eq.\,\meqref{eq:deltahom} for $m+p=k+1$ can be achieved by showing that $Z_n\subset T$ for all $n\ge 0$ and can be proved by a similar induction argument and the inductive hypothesis on $m+p$, completing the induction on $m+p$.
Hence $\Deltab$  is  an algebra  homomorphism with respect to  $\shap_{\le 0}$.
\end{proof}

Consider the linear maps
\vsb
\begin{equation}
\varepsilon_{\le 0}: \calhc\longrightarrow \Q,
\quad \varepsilon_{\le 0}(x):=\left\{\begin{array}{ll}
  1, x={\bf1},\\
  0, {\rm otherwise}.
  \end{array}\right.
\mlabel{eq:counitle0}
\vsb
\end{equation}
and
\vsc
$$u_{\leq 0}: \Q\longrightarrow \calhc, \quad u_{\leq 0}(q):=q{\bf 1}, q\in \Q.
$$

\begin{theorem}
\mlabel{prop:calhcHopf} \mlabel{thm:le0diffHopf}
The sextuple $(\calhc, \shap_{\le 0} ,u_{\leq 0}, \Deltab, \varepsilon_{\leq 0}, J_{\le 0})$  is a connected differential graded bialgebra, and hence a differential graded Hopf algebra, with the antipode $S_{\le 0}$ defined by Eq.\,\meqref{eq:connantipode}.
\end{theorem}
\begin{proof}
We first verify that the coproduct preserves the grading, for which we apply the induction on the grading $k$. First of all, since $G_0=\{{\bf1}\}$, it is clear that
\vsb
$$\Deltab({\bf 1})={\bf 1}\otimes {\bf 1}\in G_0\otimes G_0.
$$

Next assume  $\Deltab(G_k)\subset\bigoplus_{i+j=k}\Q ( G_i\otimes G_j)$ for a given $k\geq 0$. Then for $[s_1,\cdots\hskip-1mm, s_m]\in G_{k+1}$, where  $m-(s_1+\cdots+s_m)=k+1$, we proceed in two cases.

\noindent
{\bf Case 1:} Let $s_1< 0$. By the definition of $\Deltab$,
\vsb
\begin{align*}
\Deltab([s_1,\cdots\hskip-1mm, s_m])&=\Deltab J_{\le 0}([s_1+1,\cdots\hskip-1mm, s_m])
=( \id \otimes J_{\le 0}+J_{\le 0}\otimes \id)\Deltab([s_1+1,\cdots\hskip-1mm, s_m]).
\end{align*}
Since $[s_1+1,\cdots\hskip-1mm, s_m]\in G_k$, the induction hypothesis gives
\vsa
$$\Deltab([s_1+1,\cdots\hskip-1mm, s_m])\in\bigoplus_{i+j=k}\Q ( G_i\otimes G_j).
$$
So
\vsb
$$( \id \otimes J_{\le 0}+J_{\le 0}\otimes \id)\Deltab([s_1+1,\cdots\hskip-1mm, s_m])\in\bigoplus_{i+j=k+1}\Q ( G_i\otimes G_j).
\vsa
$$

\noindent
{\bf Case 2:}  Let $s_1=0$. Then the coproduct satisfies
\vsb
$$\Deltab([0, s_2,\cdots\hskip-1mm, s_m])=\Big(\Deltab([0])\Big)\shap_{\le 0} \Big(\Deltab([s_2,\cdots\hskip-1mm, s_m])\Big),
$$
where $[s_2,\cdots\hskip-1mm, s_m]\in G_k$. By the induction  hypothesis,
$\Deltab([s_2,\cdots\hskip-1mm, s_m])$ is in $\bigoplus_{i+j=k}\Q ( G_i\otimes G_j);$
while
\vsb
$$ \Deltab([0])={\bf1} \otimes [0]+[0]\otimes {\bf1}\in (G_0\otimes G_1\oplus G_1 \otimes G_0).
$$
Thus
\vsb
$$\Big(\Deltab([0])\Big)\shap_{\le 0} \Big(\Deltab([s_2,\cdots\hskip-1mm, s_m])\Big)\in\bigoplus_{i+j=k+1}\Q ( G_i\otimes G_j).
\vsa
$$

Furthermore, for a basis element $[\vec s]\in \calhc$,
\vsa
$$(\varepsilon_{\leq 0}\otimes \id)\Deltab([\vec s])=\Big(\varepsilon_{\leq 0}({\bf 1})\Big)[\vec s]=[\vec s],\quad (\id\otimes \varepsilon_{\leq 0})\Deltab([\vec s])=[\vec s]\Big(\varepsilon_{\leq 0}({\bf 1})\Big)=[\vec s],
$$
showing that $\varepsilon_{\leq 0}$ is the counit.
In addition, for a basis element $[\vec s]\in  \calhc$ and scalar $q\in \Q$, we have
$$\Big(\shap_{\le 0}(\id\otimes u_{\le 0})\Big)([\vec s]\otimes q)=q[\vec s],\quad \Big(\shap_{\le 0}(u_{\le 0}\otimes \id)\Big)(q\otimes [\vec s])=q[\vec s],
$$
showing that $u_{\le 0}$ is the unit.
Then by  Lemma  \mref{lem:D<0homomorphism}, $(\calhc, \shap_{\le 0},  u_{\leq 0}, \Deltab, \varepsilon_{\leq 0})$ is a connected graded  bialgebra, and hence a graded Hopf algebra \cite{Man}.

Finally, Proposition \mref{lem:le0differentialA} and Proposition \mref{prop:uniquedeltale0} give that $J_{\le 0}$ is a derivation and coderivation. Hence we have a connected graded differential bialgebra, and thereby a graded differential Hopf algebra by Theorem~\mref{thm:autoSd=dS}.
\end{proof}

\subsection{Hopf algebras on the graded duals $\calhg$ and $\calhf$} We now consider the Hopf algebra structure induced on the dual spaces.

Since $\Q\cong \Q^*$, to simplify the notation, we identify $1=1^*$ for $1\in\Q$.
Let $V$ be a graded vector space
$V=\bigoplus\limits_{k\ge 0} V_k,$
with $V_k$ finite dimensional. Then
 the  graded dual $V^*$  of $V$ is
$V^*=\bigoplus\limits_{k\ge 0} V_k^*.$

\begin{lemma}
Let $(H, m, u, \Delta, \vep, S, d)$ be a graded differential Hopf algebra with finite-dimensional homogeneous components $V_k$. Let $\Delta^*, \vep^*, m^*, u^*, S^*$ and $d^*$ be the linear duals. Then  the septuple $(H^*, \Delta^*, \vep^*, m^*, u^*, S^*, d^*)$ is also a graded differential Hopf algebra. \label{item:4}
\end{lemma}
\begin{proof}
It follows from \cite{Abe, Kc, Swee} that $(H^*, \Delta^*, \vep^*, m^*, u^*, S^*)$ is a graded Hopf algebra.

For $x, y\in H^*$ and $z\in H$, we have
\begin{equation*}
\begin{split}
\bigg(\Big(\Delta^*(d^*\otimes\id+\id\otimes d^*)\Big)(x^*\otimes y^*)\bigg)(z)&=\Big( (d^*\otimes\id+\id\otimes d^*)(x^*\otimes y^*)\Big)\big(\Delta (z)\big)\\
&=(x^*\otimes y^*)\big((d\otimes\id+\id\otimes d)\Delta (z)\big)\\
&\overset{\textcircled{1}}{=}(x^*\otimes y^*)\big(\Delta d(z)\big)\\
&=\Big(\Delta^*(x^*\otimes y^*)\Big)\big(d(z)\big)\\
&=\big((d^*\Delta^*)(x^*\otimes y^*)\big)(z),
\end{split}
\end{equation*}
where $\textcircled{1}$ because of $d$ coderivation for $\Delta$. Thus $d^*$ is a derivation with respect to $\Delta^*$. Similarly,
\begin{align*}
\big((d^*\otimes\id+\id\otimes d^*)m^*(w)\big)(u\otimes v)&=w\Big(\big(m(d\otimes \id+\id\otimes d)\big)(u\otimes v)\Big)\\
&=w\Big(d\big( m (u\otimes v)\big)\Big)\\
&=\big(m^*d^*(w)\big)(u\otimes v)
\end{align*}
for $w\in H^*$, $u, v\in H$. Then $d^*$ is a coderivation for $m^*$. Thus, $(H^*, \Delta^*, \vep^*, m^*, u^*)$ is a connected graded differential bialgebra, and hence a differential Hopf algebra by Theorem \ref{thm:autoSd=dS}.
\end{proof}

\begin{coro}
\begin{enumerate}
\item 
The septuple  $(\calhg, \Deltae, \varepsilon_{\geq 1}^*, \shap_{\ge 1}^*, u_{\ge 1}^*, (S_{\ge 1})^*, J_{\ge 1}^*)$ is a differential Hopf algebra.
\item 
The septuple $(\calhf, \Deltab^*, \varepsilon_{\leq 0}^*, \shap_{\le 0}^*, u_{\leq 0}^*,  (S_{\le 0})^*, J_{\le 0}^*)$ is a differential Hopf algebra.
\end{enumerate}
\mlabel{co:ge1dual}
\end{coro}

We display details of some dual maps for later applications.
\begin{eqnarray}
	&u_{\ge 1}^*:\calhg\longrightarrow\Q,
	\quad u_{\ge 1}^*(x^*):=\left\{\begin{aligned}
		1, &\quad x^*={\bf1}^*,\\
		0, &\quad {\rm otherwise},
	\end{aligned}\right. \quad x\in \calhd.&
	\mlabel{eq:counitge1dual}
\\
&	u_{\le 0}^*:\calhf\longrightarrow\Q,
	\quad u_{\le 0}^*(x^*):=\left\{\begin{array}{ll}
		1, &x^*={\bf1}^*,\\
		0, &{\rm otherwise}.
	\end{array}\right. \quad x^*\in \calhf. &
	\mlabel{eq:counitle0dual}
\\
 \mlabel{lem:Jle0*}
&J^{\ast}_{\le 0}: \calhf\longrightarrow \calhf, \quad J_{\le 0}^*({\bf 1}^*)=0, \quad
	J_{\le 0}^*([s_1,\vec s]^*)=\left\{\begin{array}{ccc}
		[s_1+1, \vec s]^*,& s_1<0;\\
		0, & s_1=0.
	\end{array}\right.&
\end{eqnarray}

\section{Dualities of two pairs of differential Hopf algebras}
\mlabel{s:iso}

In this section, we present a dual pair of isomorphisms of differential Hopf algebras $\varphi:\calhd\longrightarrow \calhf$ and $\varphi^*:\calhc\longrightarrow \calhg$. The proof is reduced to two coalgebra homomorphisms.

\subsection{The \rdual as an isomorphism of differential Hopf algebras}
\mlabel{ss:main}
Motivated by the functional equation for the Riemann zeta function, we give an algebraic analog.

\begin {defn} \label{d:dualmap}
The {\bf \rdual} is the linear map
$$\varphi:\calhd\longrightarrow \calhf, \quad \left\{ \begin{array}{l} \varphi({\bf 1}):={\bf 1}^{\ast},\\
\varphi([s_1,\cdots\hskip-1mm, s_k]):=[1-s_1, \cdots\hskip-1mm, 1-s_k]^{\ast}.
\end{array}
\right .$$
\end{defn}
The term \rdual is due to the symmetry $[s_1,\cdots\hskip-1mm, s_k]\leftrightarrow [1-s_1, \cdots\hskip-1mm, 1-s_k]^{\ast}$
similar to the symmetry $s\leftrightarrow 1-s$ in the functional equation of the Riemann zeta function.

From the bijection $\Z_{\ge 1}\longrightarrow \Z_{\le 0}, s\mapsto 1-s$, we see that  $\varphi$  is an isomorphism between  graded  vector spaces.
Then by transport of structures, the linear operators $\{\delta_i:\calhd\to \calhd\,|\, i\geq 1\}$, defined in (\mref{eq:deltadef}), induces a family of linear operators $\{\tilde{\delta}_{i}:=\varphi\delta_i\varphi^{-1}\, |\, i\ge 1\}$ on $\calhf$. Explicitly,
\begin{equation}
\label{eq:tildedelta}
\begin{split}
\tilde{\delta}_{i}:\calhf&\longrightarrow \calhf\\
{\bf1}^*&\mapsto0,\\
[s_1,\cdots\hskip-1mm, s_k]^*&\mapsto\left\{\begin{array}{ccc}
\sum_{j=1}^i(1-s_j)[s_1,\cdots\hskip-1mm, s_j-1,\cdots\hskip-1mm, s_k]^*, & k\ge i,\\
0,& i>k.
\end{array}\right.
\end{split}
\end{equation}
Moreover, let  $\tilde{A}=\id$  or  $\tilde{\delta_j}$, which is a linear map on $\calhf$.  Similar to Eq.\,\meqref{eq:shiftten}, for $i\ge 1$, the {\bf shifted tensor} is defined by
\begin{equation*}
\begin{split}
\tilde{A}\cks \tilde{\delta_i}: \calhf\otimes \calhf&\longrightarrow \calhf\otimes \calhf\\
[\vec s]^*\otimes [\vec t]^*&\mapsto
\tilde{A}([\vec s]^*)\otimes \delta_{i-\dep(\vec s)}([\vec t]^*),
\end{split}
\end{equation*}
where we also take $\tilde{\delta}_i=0$  when  $i\le 0$ for convenience.  Similar to Eq.\,(\ref{eq:pij}), denote $$\tilde{p}_k=:\tilde{\delta}_k-\tilde{\delta}_{k-1},\, k\in\Z.$$
A direct  calculation gives

\begin{lemma}
Let $i, j\in\Z_{\ge 1}$. The following identities hold.
\begin{enumerate}
\item \mlabel{eq:tildepcomm}
$\tilde{\delta}_i \tilde{\delta}_j=\tilde{\delta}_j \tilde{\delta}_i.
$
Equivalently,
$\tilde{p}_i \tilde{p}_j=\tilde{p}_j \tilde{p}_i.
$
Moreover, $\tilde{\delta}_i \tilde{p}_j=\tilde{p}_j \tilde{\delta}_i$.
\item
$\tilde{\delta}_i \varphi =\varphi \delta_i.
$
\mlabel{eq:deltaphicomm}
\item
\mlabel{eq:shiftdeltaphicomm}
$(\varphi\otimes \varphi)(A\cks \delta_i)=(\tilde{A}\cks \tilde{\delta}_i)(\varphi\otimes \varphi),
$
for  $(A, \tilde{A})=(\id, \id)$  or $(\delta_j, \tilde{\delta}_j)$.
\end{enumerate}
\mlabel{lem:p1phieqphipi}
\end{lemma}

In addition, by a direct checking using\, \meqref{lem:Jle0*}, we obtain
\begin{lemma}\mlabel{lem:varphicommJge1}
 The linear dual $J_{\leq 0}^*:\calhf\to \calhf$ of $J_{\leq 0}$ coincides with the linear operator $\varphi J_{\geq 1} \varphi^{-1}$, that is,
$\varphi J_{\ge 1}=J_{\le 0}^*\varphi.$
\end{lemma}

The \rdual
$\varphi:\calhd\to \calhf$, as a graded vector space isomorphism, induces a dual isomorphism $\varphi ^*:\calhc\to \calhg$. Then Lemma~\mref{lem:varphicommJge1} gives

\begin{lemma}
\mlabel{lem:IJphicomm}
As operators from $\calhc$ to $\calhg$, we have
$J_{\ge 1}^{\ast}\varphi ^*=\varphi^* J_{\le 0}.$
\end{lemma}

Now we state the main theorem on the duality of differential Hopf algebras.

\begin{theorem}
\mlabel{t:dhaiso}
\begin{enumerate}
\item \mlabel{i:iso1}
The \rdual $\varphi$ is an isomorphism between differential Hopf algebras
$$\varphi :(\calhd, \shap_{\ge 1} , u_{\ge 1}, \Deltaa, \varepsilon_{\ge 1}, S_{\ge 1}, J_{\ge 1})\longrightarrow (\calhf, \Deltab^*, \varepsilon_{\leq 0}^*, \shap_{\le 0}^*, u_{\leq 0}^*,  (S_{\le 0})^*, J_{\le 0}^*)$$
\item \mlabel{i:iso2}
The induced map $\varphi^*$ is an isomorphism between differential Hopf algebras
$$\varphi^*:(\calhc, \shap_{\le 0} , u_{\leq 0}, \Deltab, \varepsilon_{\leq 0}, S_{\le 0}, J_{\le 0})\longrightarrow (\calhg, \Deltae, \varepsilon_{\geq 1}^*, \shap_{\ge 1}^*, u_{\ge 1}^*, (S_{\ge 1})^*, J_{\ge 1}^*)
$$
\end{enumerate}
\end{theorem}

\begin{proof}
\eqref{i:iso1} The main body of the proof is to show that that $\varphi$ is a bialgebra isomorphism. Due to its lengthy proof, it is divided into the following two parts whose details are postponed to the next two subsections, Sections \mref{ss:coisocg} and \mref{ss:coisodf}.
\begin{enumerate}[label=(\alph*)]
\item The property that $\varphi:\calhd \to \calhf$ is a bialgebra isomorphism amounts to that $\varphi:\calhd\to \calhf$ is an algebra isomorphism which will be proved in Corollary~\mref{coro:phialgiso}.
\item The property that $\varphi$ is a coalgebra isomorphism which will be proved in Corollary~\mref{coro:CoalgebraIso1}.
\end{enumerate}

For the rest of the proof, by Lemma \mref{lem:varphicommJge1}, $\varphi$ commutes with the differential operators $J_{\geq 1}$ and $J_{\leq 0}^*$.
By \cite[Lemma 4.0.4]{Swee}, the bialgebra isomorphism $\varphi$ is an isomorphism of Hopf algebras.

Putting everything together, we have proved that $\varphi$ is an isomorphism between differential Hopf algebras.
	
\meqref{i:iso2} This is obtained by taking the linear dual of Item~\meqref{i:iso1} and applying Corollary\,\mref{co:ge1dual}.
\end{proof}

\subsection{Coalgebra isomorphism between $\calhc$ and $\calhg$}
\mlabel{ss:coisocg}
As a part of the proof of Theorem \mref{t:dhaiso}, here we will construct a coproduct on  $\calhc$  via the coproduct $\shap_{\ge 1}^*$ on  $\calhg$  and a map $\varphi^* :\calhc\longrightarrow \calhg$, and verify that it is equal to  $\Deltab$. Then we obtain that $\varphi^*:(\calhc, \Deltab)\to (\calhg, \shap_{\ge 1}^*)$ is a coalgebra isomorphism, implying that $\varphi: \calhd\to \calhf$ is an algebra isomorphism (Corollary~\mref{coro:phialgiso}).

As a preparation, we first give a recursion of $\shap_{\ge 1}^*$.

\begin{lemma} \mlabel{lem:vec1-vecsexpension}
For a basis element $[\vec s]\in \calhc$, let
$$\shap_{\ge 1}^* ([\vec 1-\vec s]^*)=\sum A_{[\vec u]^*, [\vec v]^*}^{[\vec 1-\vec s]^*}[\vec u]^*\otimes [\vec v]^*+{\bf1}^*\otimes [\vec 1-\vec s]^*+[\vec 1-\vec s]^*\otimes {\bf 1}^*,
$$
$$\shap_{\ge 1}^* ([1, \vec 1-\vec s]^*)=\sum A_{[\vec u]^*, [\vec v]^*}^{[1,\vec 1-\vec s]^*}[\vec u]^*\otimes [\vec v]^*+{\bf1}^*\otimes [1, \vec 1-\vec s]^*+[1, \vec 1-\vec s]^*\otimes {\bf 1}^*.
$$
Then the coefficients satisfy the relation
\vsa
$$A_{[\vec u]^*, [\vec v]^*}^{[1,\vec 1- \vec s]^*}=\left\{\begin{array}{lll}
0,& u_1>1, v_1>1;\\
A_{[\vec u\,']^*, [\vec v]^*}^{[\vec 1-\vec s]^*}, & u_1=1, v_1>1;\\
A_{[\vec u]^*, [\vec v\,']^*}^{[\vec 1-\vec s]^*}, &  u_1>1, v_1=1;\\
A_{[\vec u\,']^*, [1, \vec v\,']^*}^{[\vec 1-\vec s]^*}+A_{[1,\vec u\,']^*, [\vec v\,']^*}^{[\vec 1-\vec s]^*}, & u_1=v_1=1,
\end{array}\right.
$$
with $\vec u=(u_1, \vec u\,')$, $\vec v=(v_1, \vec v\,')$.
\end{lemma}
\begin {proof} Since $\shap_{\ge 1}^*([1, \vec 1-\vec s]^*)\in\calhg\otimes \calhg$, for $[x_1, \vec x]$,  $[y_1, \vec y]\in \calhd$, by Lemma \mref {lemma:ShuffleInHGe1}, we get
\small{
\begin{align*}
&\Big(\shap_{\ge 1}^* ([1, \vec 1-\vec s]^*)\Big)\Big([x_1, \vec x]\otimes [y_1, \vec y]\Big)=[1, \vec 1-\vec s]^*\Big([x_1, \vec x]\shap_{\ge 1} [y_1, \vec y]\Big)\\
=&[1, \vec 1-\vec s]^*\Big(\sum_{i=0}^{y_1-1} \binc{x_1-1+i}{i}\Big[x_1+i, \vec x\shap_{\ge 1} [y_1-i, \vec y]\Big]\Big)+[1, \vec 1-\vec s]^*\Big(\sum_{i=0}^{x_1-1}\binc{y_1-1+i}{i}\Big[y_1+i, [x_1-i, \vec x]\shap_{\ge 1} \vec y\Big]\Big)\\
=&\left\{\begin{array}{lll}
0,  &x_1>1, y_1>1;\\
A_{[\vec x]^*, [y_1, \vec y]^*}^{[\vec 1-\vec s]^*}, &x_1=1, y_1>1;\\
A_{[x_1,\vec x]^*, [ \vec y]^*}^{[\vec 1-\vec s]^*}, &x_1>1, y_1=1;\\
A_{[\vec x]^*, [1, \vec y]^*}^{[\vec 1-\vec s]^*}+A_{[1,\vec x]^*, [\vec y]^*}^{[\vec 1-\vec s]^*}, &x_1=1, y_1=1.
\end{array}\right.
\end{align*}}
On  the other hand,
\vsa
$$\Big(\shap_{\ge 1}^* ([1, \vec 1- \vec s]^*)\Big)([x_1, \vec x]\otimes [y_1, \vec y])=A_{[x_1,\vec x]^*, [y_1,\vec y]^*}^{[1, \vec 1-\vec s]^*}.
$$
So by a change of variables, we arrive at the conclusion.
\end{proof}

Let $\psi:= (\varphi ^*)^{-1}: \calhg \to \calhc$. Now we define
\begin{equation}
\Deltad:=(\psi \otimes \psi )\shap_{\ge 1}^*\varphi ^* :\calhc\to \calhc\otimes \calhc.
\end{equation}

\begin{prop}\mlabel{prop:Deltad}
The linear  map  $\Deltad$  satisfies the follow conditions.
\begin{enumerate}
\item $\Deltad({\bf 1})={\bf 1}\otimes {\bf 1}$;
\mlabel{i:dad1}
\item $\Deltad([0])={\bf 1}\otimes [0]+[0]\otimes {\bf 1}$;
\mlabel{i:dad2}
\item $\Deltad J_{\le 0}=( \id \otimes J_{\le 0}+J_{\le 0}\otimes \id)\Deltad$;
\mlabel{i:dad3}
\item  $\Deltad([0, \vec s])=\Big(\Deltad([0])\Big)\shap_{\le 0} \Big(\Deltad([\vec s])\Big)$  for $\vec s\in \Z_{\le 0}^\ell\ (\ell\ge 0)$.
\mlabel{i:dad4}
\end{enumerate}
\end{prop}

\begin{proof}
\meqref{i:dad1}. By the definition of $\Deltad$, we have
$\Deltad ({\bf 1})=(\psi \otimes \psi)\shap_{\ge 1}^*\varphi^*({\bf 1})={\bf 1}\otimes {\bf 1}.
$

\noindent
\meqref{i:dad2}. Furthermore,
\begin{align*}
\Deltad([0])&=(\psi\otimes \psi)\shap_{\ge 1}^*\varphi([0])=(\psi\otimes \psi)\shap_{\ge 1}^*([1]^*)\\
&=(\psi \otimes \psi )({\bf 1}^*\otimes [1]^*+[1]^*\otimes {\bf 1}^*)={\bf 1}\otimes [0]+[0]\otimes {\bf 1}.
\end{align*}
\meqref{i:dad3}.
Applying  Lemmas  \mref{lem:IJphicomm}  and Corollary\,\mref{co:ge1dual}, we obtain
\begin{align*}
\Deltad J_{\le 0}&=(\psi\otimes \psi)\shap_{\ge 1}^*\varphi^* J_{\le 0}=(\psi\otimes \psi)\shap_{\ge 1}^*J_{\ge 1}^*\varphi ^*\\
&=(\psi\otimes \psi)( \id \otimes J_{\ge 1}^*+J_{\ge 1}^*\otimes \id)\shap_{\ge 1}^*\varphi ^*\\
&=( \id \otimes J_{\le 0}+J_{\le 0}\otimes \id)(\psi\otimes \psi)\shap_{\ge 1}^*\varphi ^*=( \id \otimes J_{\le 0}+J_{\le 0}\otimes \id)\Deltad.
\end{align*}
\meqref{i:dad4}
Write
$$\Deltad ([\vec s])=\sum a_{[\vec u], [\vec v]}^{[\vec s]}[\vec u]\otimes [\vec v]+{\bf1}\otimes [\vec s]+[\vec s]\otimes {\bf 1},
$$
and
$$\Deltad ([0, \vec s])=\sum a_{[\vec u], [\vec v]}^{[0,\vec s]}[\vec u]\otimes [\vec v]+{\bf1}\otimes [0, \vec s]+[0,\vec s]\otimes {\bf 1}.
$$
Then 
\begin{align*}
&\Big(\Deltad([0])\Big)\shap_{\le 0} \Big(\Deltad ([\vec s])\Big)\\
=&({\bf1}\otimes [0]+[0]\otimes {\bf 1})\shap_{\le 0}\Big(\sum a_{[\vec u], [\vec v]}^{[\vec s]}[\vec u]\otimes [\vec v]+{\bf1}\otimes [\vec s]+[\vec s]\otimes {\bf 1}\Big)\\
=&\sum a_{[\vec u], [\vec v]}^{[\vec s]}[\vec u]\otimes [0,\vec v]+\sum a_{[\vec u], [\vec v]}^{[\vec s]}[0,\vec u]\otimes [\vec v]+{\bf 1}\otimes [0,\vec s]\\
&+[0]\otimes[\vec s]+[\vec s]\otimes [0]+[0, \vec s]\otimes {\bf 1}.
\end{align*}
Since  $\Deltad([0, \vec s])=(\psi\otimes\psi)\Big(\shap_{\ge 1}^* ([1, \vec 1- \vec s]^*)\Big),$ utilizing Lemma \mref{lem:vec1-vecsexpension} gives
$$a_{[\vec u], [\vec v]}^{[0,\vec s]}=A_{\varphi([\vec u]), \varphi([\vec v])}^{\varphi([0, \vec s])}.
$$
Denote $\vec u=(u_1, \vec u\,'), \vec v=(v_1, \vec v\,')$. Then
$$a_{[u_1,\vec u\,'], [v_1,\vec v\,']}^{[0,\vec s]}=\left\{\begin{array}{ll}
0, & u_1<0, v_1<0;\\
a_{[\vec u\,'],[v_1, \vec v\,']}^{[\vec s]}, & u_1=0, v_1<0;\\
a_{[u_1, \vec u\,'],[\vec v\,']}^{[\vec s]}, & u_1<0, v_1=0;\\
a_{[\vec u\,'],[0, \vec v\,']}^{[\vec s]}+ a_{[0, \vec u\,'],[\vec v\,']}^{[\vec s]}, & u_1=0, v_1=0.
\end{array}\right.
$$
Since  $a_{[0],[\vec s]}^{[0, \vec s]}=a_{{\bf1},[\vec s]}^{[\vec s]}=1$,  $a_{[\vec s], [0]}^{[0, \vec s]}=a_{[\vec s], {\bf 1}}^{[\vec s]}=1$, we obtain the needed equality
$$\hspace{5cm} \Deltad([0,\vec s])=\Big(\Deltad([0])\Big)\shap_{\le 0} \Big(\Deltad([\vec s])\Big). \hspace{5cm}\qedhere
$$
\end{proof}

Now we apply Proposition  \mref{prop:uniquedeltale0} to reach our first conclusion on $\varphi$.

\begin{coro}
\mlabel{coro:deltab=d}
\mlabel{coro:phialgiso}
\begin{enumerate}
\item \mlabel{i:deltab=d1}
As linear maps from $\calhc$ to $\calhc\otimes \calhc$, we have
$\Deltad=\Deltab;$
\item \mlabel{i:deltab=d2}
$\varphi^*$ is a graded coalgebra isomorphism
$\varphi^*: (\calhc, \Deltab, \varepsilon_{\le 0})\longrightarrow (\calhg, \shap_{\ge 1}^*, u_{\ge 1}^*).$
\item \mlabel{i:deltab=d3}
$\varphi$ is an algebra isomorphism $\varphi :(\calhd, \shap_{\ge 1}, u_{\ge 1})\longrightarrow (\calhf, \Deltaf, \varepsilon_{\le 0}^*)$ .
\end{enumerate}
\end{coro}

\begin{proof}
\meqref{i:deltab=d1}
Proposition~\mref{prop:Deltad} shows that $\Deltad$ satisfies the same recursion as $\Deltab$ defined in Proposition~\mref{prop:uniquedeltale0}. Thus the two coincide due to the uniqueness of the recursion.

\noindent
\meqref{i:deltab=d2}
By definition, $\Deltad$ satisfies $(\varphi^*\otimes \varphi^*)\Deltad=\shap_{\ge 1}^*\varphi^*.$ Hence Item \meqref{i:deltab=d1} gives
$(\varphi^*\otimes \varphi^*)\Deltab=\shap_{\ge 1}^*\varphi^*.$
Further, by Eq.\,(\mref{eq:counitge1dual}), the linear map $u_{\ge 1}^*\varphi^* : \calhc \longrightarrow \Q$ is given by
$$u_{\ge 1}^*\varphi^*(x)=\left\{\begin{array}{ll}
1,& x={\bf1},\\
0, & \text{otherwise}.
\end{array}\right.
$$
for a basis element $x$ of $\calhc$. Hence $u_{\ge 1}^*\varphi^*=\varepsilon_{\le 0}$. Thus $\varphi^*: (\calhc, \Deltab, \varepsilon_{\le 0})\longrightarrow (\calhg, \shap_{\ge 1}^*, u_{\ge 1}^*)$ is an isomorphism between graded coalgebras.

\noindent
\meqref{i:deltab=d3} This follows from taking the graded linear dual in Item~\meqref{i:deltab=d2}.
\end{proof}

\subsection{Coalgebra isomorphism between $\calhd$ and $\calhf$}
\mlabel{ss:coisodf}
In this part, we will construct a coproduct on  $\calhd$  via the coproduct $\shap_{\le 0}^*$ on  $\calhf$  and the \rdual $\varphi :\calhd\longrightarrow \calhf$, and verify that the induced coproduct is equal to  $\Deltaa$. Then we get a coalgebra isomorphism between $(\calhd, \Deltaa)$ and the dual $(\calhf, \shap_{\le 0}^*)$. We first present some properties of  $\shap_{\le 0}^*$.

\subsubsection{A recursion of $\shap_{\le 0}^*$}
For a fixed  $\vec s\in \Z_{\le 0}^k$, write
\begin{equation}
\mlabel{eq:a}
\shap_{\le 0}^{*}([\vec s]^{*})={\bf1}^*\otimes [\vec s]^{*}+[\vec s]^{*}\otimes{\bf1}^*+\sum_{\vec u\in\Z_{\le 0}^m, \vec v\in\Z_{\le 0}^n, m,n\geq 1} a_{[\vec u]^*, [\vec v]^*}^{[\vec s]^*}[\vec u]^{*}\otimes [\vec v]^{*}.
\end{equation}

\begin{lemma}
For $[\vec s]\in \calhc$, $\vec u\in \Z_{\le 0}^m$ and $\vec v\in\Z_{\le 0}^n $,
$$a_{[\vec u]^*, [\vec v]^*}^{[\vec s]^*}=a_{\big(J_{\le 0}([\vec u])\big)^*, [\vec v]^*}^{\big(J_{\le 0}([ \vec s])\big)^*}+a_{[\vec u]^*, \big(J_{\le 0}([\vec v])\big)^*}^{\big(J_{\le 0}([ \vec s])\big)^*}
$$
\mlabel{lem:ab}
\end{lemma}

\begin{proof}
Let  $\vec s=(s_1, \vec s\,')$. By  definition, we have
$$\Big(\shap_{\le 0}^*([s_1-1,\vec s\,']^*)\Big)\Big(J_{\le 0}([\vec u])\otimes [\vec v]\Big)=a_{\big(J_{\le 0}([\vec u])\big)^*, [\vec v]^*}^{\big(J_{\le 0}([ \vec s])\big)^*}
$$
On the other hand, by Proposition \ref{lem:le0differentialA}, and Eqs.\,\meqref{lem:Jle0*} and \meqref{eq:a},
\begin{align*}
&\Big(\shap_{\le 0}^*([s_1-1,\vec s\,']^*)\Big)\Big(J_{\le 0}([\vec u])\otimes [\vec v]\Big)=([s_1-1,\vec s\,']^*)\Big(J_{\le 0}([\vec u])\shap_{\le 0} [\vec v]\Big)\\
=&[s_1-1,\vec s\,']^* \Big(J_{\le 0}([\vec u]\shap_{\le 0} [\vec v])\Big)-[s_1-1,\vec s\,']^*\Big([\vec u]\shap_{\le 0} J_{\le 0}([\vec v])\Big)\\
=&\Big(J_{\le 0}^*([s_1-1,\vec s\,']^*)\Big)([\vec u]\shap_{\le 0} [\vec v])-\Big(\shap_{\le 0}^* ([s_1-1,\vec s\,']^*)\Big)\Big([\vec u]\otimes J_{\le 0}([\vec v])\Big)\\
=&\Big(\shap_{\le 0}^*([s_1,\vec s\,']^*)\Big)([\vec u]\otimes [\vec v])-\Big(\shap_{\le 0}^* ([s_1-1,\vec s\,']^*)\Big)\Big([\vec u]\otimes J_{\le 0}([\vec v])\Big)\\
=&a_{[\vec u]^*, [\vec v]^*}^{[\vec s]^*}-a_{[\vec u]^*, \big(J_{\le 0}([\vec v])\big)^*}^{\big(J_{\le 0}([ \vec s])\big)^*}
\end{align*}
Hence we have the conclusion.
\end{proof}

For $(s_1,\vec s)\in\Z^k_{\le 0}$, we rewrite Eq. \meqref{eq:a} as
\begin{equation}
\shap_{\le 0}^*([s_1, \vec s]^*)={\bf 1}^*\otimes  [s_1,\vec s]^*+[s_1, \vec s]^*\otimes {\bf 1}+\sum_{(u_1, \vec u)\in\Z_{\le 0}^m, \vec v\in \Z_{\le 0}^n,m,n\geq 1}  a^{[s_1, \vec s]^*}_{[u_1, \vec u]^*, [\vec v]^*} [u_1, \vec u]^*\otimes [\vec v]^*.
\mlabel{eq:b}
\end{equation}
We will present the recursive properties of the coefficients $ a^{[s_1, \vec s]^*}_{[u_1, \vec u]^*, [\vec v]^*}$, in the two cases of $u_1 > s_1$ and $u_1\le s_1$, in Lemma \ref{lem:0foru1greats1} and Lemma \ref{lem:relationaaab} respectively.

\begin{lemma}
\mlabel{lem:0foru1greats1}
For $(s_1,\vec s)\in \Z_{\le 0}^k$ and $(u_1,\vec u)\in \Z_{\le 0}^m$, if $u_1>s_1$, then
$a_{[u_1, \vec u]^*, [\vec v]^*}^{[s_1,\vec s]^*}=0.$
\end{lemma}

\begin{proof}
The proof is an induction on $-s_1\geq 0$.  If $-s_1=0$, there is no term $[u_1, \vec u]^*$ in $\calhf$ with $u_1>0$. So $a_{[u_1, \vec u]^*, [\vec v]^*}^{[0,\vec s]^*}=0$  when  $u_1>0$.

Assume that, for a given $k\geq 0$, the statement holds for $-s_1=k$,  that is, for  $u_1>-k$,
there is
$$a_{[u_1, \vec u]^*, [\vec v]^*}^{[-k,\vec s]^*}=0.
$$
Then consider the case $-s_1=k+1$ and apply a second induction on $-u_1\geq 0$. Since $-k-1$ is less than $0$, we have
$$\Big(\shap_{\le 0}^*([-k-1,\vec s]^*)\Big)([0,\vec u]\otimes [\vec v])=[-k-1,\vec s]^*([0, \vec u\shap_{\le 0} \vec v])=0.
$$
This means that $a_{[u_1,\vec u]^*, [\vec v]^*}^{[-k-1,\vec s]^*}=0$ for $-u_1=0$.
Assume that $a_{[u_1, \vec u]^*, [\vec v]^*}^{[-k-1,\vec s]^*}=0$  when $-u_1=\ell-1$ for a given $\ell$ with $1\le \ell(< k+2)$. By Lemma \mref{lem:ab},
$$a_{[u_1,\vec u]^*, [\vec v]^*}^{[-k,\vec s]^*}=a_{[u_1-1,\vec u]^*, [\vec v]^*}^{[-k-1, \vec s]^*}+a_{[u_1,\vec u]^*, \big(J_{\le 0}([\vec v])\big)^*}^{[-k-1, \vec s]^*}
$$
So the inductive hypothesis implies
$$a_{[-\ell,\vec u]^*, [\vec v]^*}^{[-k-1, \vec s]^*}=a_{[-\ell+1,\vec u]^*, [\vec v]^*}^{[-k,\vec s]^*}-a_{[-\ell+1,\vec u]^*, \big(J_{\le 0}([\vec v])\big)^*}^{[-k-1, \vec s]^*}=0.$$
 That is $a_{[u_1,\vec u]^*, [\vec v]^*}^{[-k-1, \vec s]^*}=0$ for $-u_1=\ell$. This completes the induction on $-u_1$ and thereby on $-s_1$.
\end{proof}
\begin{lemma}
\mlabel{lem:relationaaab}
For $(s_1,\vec s)\in \Z_{\le 0}^k$ and $(u_1,\vec u)\in \Z_{\le 0}^m$, if  $ u_1\le s_1$, then
\begin{equation}
(s_1-u_1)a_{[u_1,\vec u]^*, [\vec v]^*}^{[s_1,\vec s]^*}={u_1}a_{[u_1+1,\vec u]^*, \big(J_{\le 0}([\vec v])\big)^*}^{[s_1,\vec s]^*},\quad
a_{[u_1,\vec u]^*, [\vec v]^*}^{[s_1,\vec s]^*}=\frac{1-s_1}{1-u_1}a_{[u_1-1, \vec u]^*, [\vec v]^*}^{[s_1-1,\vec s]^*}
\mlabel{e:relationaaab}
\end{equation}
\end{lemma}

\begin{proof}
We proceed by a double induction first on $-s_1\ge 0$  and then on  $-u_1\geq -s_1$.

\noindent
{\bf Step 1.} First let $s_1=0$. Then the first equality of Eq.~\meqref{e:relationaaab} is
\begin {equation}
\mlabel {eq:aa}
(0-u_1)a_{[u_1,\vec u]^*, [\vec v]^*}^{[0,\vec s]^*}={u_1}a_{[u_1+1,\vec u]^*, \big(J_{\le 0}([\vec v])\big)^*}^{[0,\vec s]^*}
\end{equation}
which is obvious for $u_1=0$.
By Eq. \meqref{eq:a},
$$\shap_{\le 0}^*([0,\vec s]^*)={\bf 1}^*\otimes [0,\vec s]^*+[0,\vec s]^*\otimes {\bf 1}^*+ \sum a_{[u_1,\vec u]^*, [\vec v]^*}^{[0,\vec s]^*}[u_1, \vec u]^*\otimes [\vec v]^*.
$$
Then  for  $(u_1,\vec u)\in \Z_{\le 0}^m$,\ $u_1\le -1$,  $\vec v\in \Z_{\le 0}^n$, on the one hand,
$$ \Big(\shap_{\le 0}^*([0,\vec s]^*)\Big)([u_1, \vec u]\otimes [\vec v])=a_{[u_1,\vec u]^*, [\vec v]^*}^{[0,\vec s]^*}
$$
On  the other hand, by Proposition \ref{lem:le0differentialA}, Eq.\,\meqref{lem:Jle0*}, Eq.\,\meqref{eq:b} and the definition of $\shap_{\le 0}^*$,
\begin{align*}
&\Big(\shap_{\le 0}^*([0,\vec s]^*)\Big)([u_1, \vec u]\otimes [\vec v])
=[0,\vec s]^*\Big([u_1,\vec u]\shap_{\le 0} [\vec v]\Big)
\\
=&[0,\vec s]^*\Big(J_{\le 0}([u_1+1,\vec u])\shap_{\le 0} [\vec v]\Big)\\
=&[0,\vec s]^*\Big(J_{\le 0}\Big([u_1+1, \vec u]\shap_{\le 0} [\vec v]\Big)\Big)-[0,\vec s]^*\Big([u_1+1,\vec u]\shap_{\le 0} J_{\le 0}([\vec v])\Big)\\
=&\Big(J_{\le 0}^*[0,\vec s]^*\Big)\Big([u_1+1, \vec u]\shap_{\le 0} [\vec v]\Big)-\Big(\shap_{\le 0}^*([0,\vec s]^*)\Big)\Big([u_1+1,\vec u]\otimes J_{\le 0}([\vec v])\Big)\\
=&-\Big(\shap_{\le 0}^*([0,\vec s]^*)\Big)\Big([u_1+1,\vec u]\otimes J_{\le 0}([\vec v])\Big)\\
=&-a_{[u_1+1, \vec u]^*, \big(J_{\le 0}([\vec v])\big)^*}^{[0, \vec s]^*}
\end{align*}
Hence
$$
(0-u_1)a_{[u_1,\vec u]^*, [\vec v]^*}^{[0,\vec s]^*}=u_1a_{[u_1+1, \vec u]^*, \big(J_{\le 0}([\vec v])\big)^*}^{[0, \vec s]^*}
$$
and we have verified Eq.\,(\mref {eq:aa}), namely the first equality in Eq.~\meqref{e:relationaaab}, when $s_1=0$.

By  Lemma  \mref{lem:0foru1greats1}, we have $a_{[0,\vec u]^*, \big(J_{\le 0}([\vec v])\big)^*}^{[-1, \vec s]^*}=0$. So by  Lemma  \mref{lem:ab},
$$a_{[-1,\vec u]^*, [\vec v]^*}^{[-1, \vec s]^*}=a_{[0,\vec u]^*, [\vec v]^*}^{[0,\vec s]^*}-a_{[0, \vec u]^*, \big(J_{\le 0}([\vec v])\big)^*}^{[-1, \vec s]^*}=a_{[0,\vec u]^*, [\vec v]^*}^{[0,\vec s]^*}
$$
This is the second equality for $s_1=0$ and $u_1=0$. Assume that this equality holds for  $u_1=-\ell$ for a given $\ell\geq 0$. Then we have
\begin{equation}
a_{[-\ell,\vec u]^*, [\vec v]^*}^{[0, \vec s]^*}=\frac{1}{1+\ell}a_{[-\ell-1,\vec u]^*, [\vec v]^*}^{[-1, \vec s]^*} \mlabel{eq:s0ulaa}
\end{equation}
Then by the induction hypothesis and Eqs. (\mref{eq:aa})--(\mref{eq:s0ulaa}), we obtain
\begin{align*}
a_{[-\ell-1,\vec u]^*, [\vec v]^*}^{[0,\vec s]^*}&=a_{[-\ell-2,\vec u]^*, [\vec v]^*}^{[-1,\vec s]^*}+a_{[-\ell-1,\vec u]^*, \big(J_{\le 0}([\vec v])\big)^*}^{[-1,\vec s]^*}\\
&\overset{(\mref{eq:s0ulaa})}{=}a_{[-\ell-2,\vec u]^*, [\vec v]^*}^{[-1, \vec s]^*}+(1+\ell)a_{[-\ell,\vec u]^*, \big(J_{\le 0}([\vec v])\big)^*}^{[0, \vec s]^*}\\
&\overset{(\mref{eq:aa})}{=}a_{[-\ell-2,\vec u]^*, [\vec v]^*}^{[-1,\vec s]^*}-(1+\ell)a_{[-\ell-1,\vec u]^*, [\vec v]^*}^{[0,\vec s]^*}
\end{align*}
Hence,
$$a_{[-\ell-1,\vec u]^*, [\vec v]^*}^{[0,\vec s]^*}=\frac{1}{\ell+2}a_{[-\ell-2,\vec u]^*, [\vec v]^*}^{[-1,\vec s]^*}
$$
This completes the double induction when $s_1=0$.
\smallskip

\noindent
{\bf Step 2. }
Next assume that Eq.~\meqref{e:relationaaab} holds for $s_1=-m$ and $u_1\le -m$, that is,
\begin{equation}
(-m-u_1)a_{[u_1,\vec u]^*, [\vec v]^*}^{[-m,\vec s]^*}={u_1}a_{[u_1+1,\vec u]^*, \big(J_{\le 0}([\vec v])\big)^*}^{[-m,\vec s]^*},\quad a_{[u_1,\vec u]^*, [\vec v]^*}^{[-m, \vec s]^*}=\frac{1+m}{1-u_1}a_{[u_1-1, \vec u]^*, [\vec v]^*}^{[-m-1, \vec s]^*}\mlabel{eq:aabb}
\end{equation}

Now consider $-s_1=m+1$ (obviously  $-s_1 \geq 1$). We apply induction on $-u_1\geq -s_1=m+1$.

For the initial step $-u_1=-s_1=m+1$, note that by  Lemma  \mref{lem:0foru1greats1},  $a_{[-m,\vec u]^*, [\vec v]^*}^{[-m-1, \vec s]^*}=0$. So both sides of the first equality in Eq. \meqref{e:relationaaab} are zero. Further, by Lemma \ref{lem:ab},
$$a_{[-m-1,\vec u]^*, [\vec v]^*}^{[-m-1, \vec s]^*}= a_{[-m-2, \vec u]^*, [\vec v]^*}^{[-m-2, \vec s]^*}+a_{[-m-1, \vec s]^*, \big(J_{\le 0}([\vec v])\big)^*}^{[-m-2, \vec s]^*}$$
This gives the second equality in Eq. \meqref{e:relationaaab} since  $a_{[-m-1, \vec s]^*, \big(J_{\le 0}([\vec v])\big)^*}^{[-m-2, \vec s]^*}=0$ by Lemma  \mref{lem:0foru1greats1}.

Next assume that, for a given $\ell\geq m+1$, the lemma holds for  $-u_1=\ell$, that is,
\begin{equation}
(-m-1+\ell)a_{[-\ell,\vec u]^*, [\vec v]^*}^{[-m-1,\vec s]^*}={-\ell}a_{[-\ell+1,\vec u]^*, \big(J_{\le 0}([\vec v])\big)^*}^{[-m-1,\vec s]^*},\quad a_{[-\ell,\vec u]^*, [\vec v]^*}^{[-m-1, \vec s]^*}=\frac{2+m}{1+\ell}a_{[-\ell-1, \vec u]^*, [\vec v]^*}^{[-m-2, \vec s]^*}\mlabel{eq:abul}
\end{equation}
Then for $-u_1=\ell+1$,  by  Lemma  $\mref{lem:ab}$  and Eq. (\mref{eq:aabb}) we have
\begin{align*}
a_{[-\ell-1,\vec u]^*, [\vec v]^*}^{[-m-1,\vec s]^*}&=a_{[-\ell,\vec u]^*, [\vec v]^*}^{[-m,\vec s]^*}-a_{[-\ell,\vec u]^*, \big(J_{\le 0}([\vec v])\big)^*}^{[-m-1,\vec s]^*}
=\frac{1+m}{1+\ell} a_{[-\ell-1, \vec u]^*, [\vec v]^*}^{[-m-1, \vec s]^*}-a_{[-\ell,\vec u]^*, \big(J_{\le 0}([\vec v])\big)^*}^{[-m-1,\vec s]^*}
\end{align*}
Collecting similar terms gives the first equality in Eq.~\meqref{e:relationaaab}:
\begin{equation}
(\ell-m)a_{[-\ell-1, \vec u]^*, [\vec v]^*}^{[-m-1, \vec s]^*}=-(\ell+1)a_{[-\ell, \vec u]^*, \big(J_{\le 0}([\vec v])\big)^*}^{[-m-1, \vec s]^*}  \mlabel{eq:abu-1}
\end{equation}
On the other hand, by  Lemma  $\mref{lem:ab}$  and  Eqs.  (\mref{eq:abul})--(\mref{eq:abu-1}), we have
{\small
\begin{align*}
a_{[-\ell-1,\vec u]^*, [\vec v]^*}^{[-m-1, \vec s]^*}&=a_{[-\ell-2,\vec u]^*, [\vec v]^*}^{[-m-2,\vec s]^*}\!+\!a_{[-\ell-1, \vec u]^*, \big(J_{\le 0}([\vec v])\big)^*}^{[-m-2,\vec s]^*}\\
&=a_{[-\ell-2,\vec u]^*, [\vec v]^*}^{[-m-2,\vec s]^*}\!+\!\frac{1+\ell}{2+m}a_{[-\ell,\vec u]^*, \big(J_{\le 0}([\vec v])\big)^*}^{[-m-1, \vec s]^*}\\
&=a_{[-\ell-2,\vec u]^*, [\vec v]^*}^{[-m-2,\vec s]^*}\!+\!\frac{m-\ell}{2+m}a_{[-\ell-1, \vec u]^*, [\vec v]^*}^{[-m-1, \vec s]^*}
\end{align*}
}
Then collecting similar terms gives the second equality in Eq. \meqref{e:relationaaab}:
$$a_{[-\ell-1,\vec u]^*, [\vec v]^*}^{[-m-1, \vec s]^*}=\frac{m+2}{\ell+2}a_{[-\ell-2,\vec u]^*, [\vec v]^*}^{[-m-2, \vec s]^*}
$$

Now the double induction is completed.
\end{proof}

\begin{lemma}
\mlabel{lem:mcomparei}
Denote $0_i:=(0,\cdots\hskip-1mm,0)\in \Z_{\le 0}^i$. Let $i\in \Z _{\ge 1}$, $(0_{i}, \vec s)\in \Z_{\le 0}^{k+i}$, $ \vec u=(u_1,\cdots\hskip-1mm,u_m)\in \Z_{\le 0}^{m}$ and $(v_1,\vec v)\in \Z_{\le 0}^n$.
\begin{enumerate}
\item \mlabel{i:comp1}
If $m< i$, then
$$a_{[\vec u]^*, [v_1,\vec v]^*}^{[0_i, \vec s]^*}=\left\{\begin{array}{lll}
1, & [\vec u]=[0_m], [v_1,\vec v]=[0_{i-m}, \vec s],\\
0, &\text{ otherwise}.
\end{array}\right.
$$
\item \mlabel{i:comp2}
If $m=i$, then
$$a_{[\vec u]^*, [v_1,\vec v]^*}^{[0_i, \vec s]^*}=\left\{\begin{array}{lll}
(-1)^{u_1+\cdots+u_i}, &[v_1+u_1+\cdots+u_i,\vec v]=[\vec s],\\
0, &\text{ otherwise}.
\end{array}\right.
$$
\item \mlabel{i:comp3}
If $m>i$, taking $\vec u\,'=(u_{i+1}, \cdots\hskip-1mm, u_m)$, then
\begin{equation}
a_{[u_1,\cdots\hskip-1mm, u_{i}, \vec u\,']^*,[v_1, \vec v]^* }^{[0_i, \vec s]^*}=(-1)^{u_1+\cdots+u_i}a_{[\vec u\,']^*, [v_1+u_1+\cdots+u_i, \vec v]^*}^{[\vec s]^*}
\mlabel{eq:cd}
\end{equation}
\end{enumerate}
\end{lemma}

\begin{proof}
We first give a preparational identity.
\begin{equation}\label{eq:0i}
\Big(\shap_{\le 0}^*([0_{i}, \vec s]^*)\Big)\Big([u_1,\cdots\hskip-1mm, u_m]\otimes [v_1,\vec v]\Big)=(-1)^{u_1}[0_{i}, \vec s]^*\Big(\Big[0,[u_2,\cdots\hskip-1mm,u_m]\shap_{\le 0} [v_1+u_1, \vec v]\Big]\Big).
\end{equation}
We prove by induction on $-u_1\geq 0$. For $-u_1=0$, by the definition of $\shap_{\le 0}$, we have
$$\Big(\shap_{\le 0}^*([0_{i}, \vec s]^*)\Big)\Big([0,u_2,\cdots\hskip-1mm, u_m]\otimes [v_1,\vec v]\Big)=[0_{i}, \vec s]^*\Big(\Big[0,[u_2,\cdots\hskip-1mm,u_m]\shap_{\le 0} [v_1+u_1, \vec v]\Big]\Big).
$$
Assume that for $u_1=-\ell$,
\begin{equation}\label{eq:0ihypo}
\Big(\shap_{\le 0}^*([0_{i}, \vec s]^*)\Big)\Big([-\ell,\cdots\hskip-1mm, u_m]\otimes [v_1,\vec v]\Big)=(-1)^{-\ell}[0_{i}, \vec s]^*\Big(\Big[0,[u_2,\cdots\hskip-1mm,u_m]\shap_{\le 0} [v_1-\ell, \vec v]\Big]\Big).
\end{equation}
Then for $u_1=-\ell-1<0$, we have
\begin{align*}
&\Big(\shap_{\le 0}^*([0_{i}, \vec s]^*)\Big)\Big([-\ell-1,\cdots\hskip-1mm, u_m]\otimes [v_1,\vec v]\Big)=[0_{i}, \vec s]^*\Big([-\ell-1,\cdots\hskip-1mm, u_m]\shap_{\le 0} [v_1,\vec v]\Big)\\
=&[0_{i}, \vec s]^*\Big(J_{\le 0}([-\ell,\cdots\hskip-1mm, u_m])\shap_{\le 0} [v_1, \vec v]\Big)\\
=&[0_{i}, \vec s]^*\bigg(J_{\le 0}\Big([-\ell,\cdots\hskip-1mm, u_m]\shap_{\le 0} [v_1, \vec v]\Big)\bigg)-[0_{i}, \vec s]^*\Big([-\ell,\cdots\hskip-1mm, u_m]\shap_{\le 0}J_{\le 0}([v_1,\vec v])\Big)\\
=&0-[0_{i}, \vec s]^*\Big([-\ell,\cdots\hskip-1mm, u_m]\shap_{\le 0} [v_1-1,\vec v]\Big)=-\Big(\shap_{\le 0}^*([0_{i}, \vec s]^*)\Big)\Big([-\ell,\cdots\hskip-1mm, u_m]\otimes [v_1-1,\vec v]\Big)\\
\overset{(\ref{eq:0ihypo})}{=}&-(-1)^{-\ell}[0_{i}, \vec s]^*\Big(\Big[0,[u_2,\cdots\hskip-1mm,u_m]\shap_{\le 0} [v_1-1-\ell, \vec v]\Big]\Big)\\
=&(-1)^{-\ell-1}[0_{i}, \vec s]^*\Big(\Big[0,[u_2,\cdots\hskip-1mm,u_m]\shap_{\le 0} [v_1-\ell-1, \vec v]\Big]\Big).
\end{align*}
This completes the induction. 

We next prove the three statements in the lemma.

\noindent
\meqref{i:comp1} If  $m< i$,  then applying Eq.\,(\ref{eq:0i}) repeatedly gives
$$\Big(\shap_{\le 0}^*([0_{i}, \vec s]^*)\Big)\Big([u_1,\cdots\hskip-1mm, u_m]\otimes [v_1,\vec v]\Big)
=(-1)^{u_1+\cdots+u_m}[0_{i},\vec s]^*([0_{m},v_1+u_1+\cdots+u_m, \vec v]).
$$
Since $u_1, \cdots\hskip-1mm, u_m \le 0$, $v_1\le 0$, we have $[0_{i},\vec s]^*([0_{m},v_1+u_1+\cdots+u_m, \vec v])\not =0$ only if $u_1=\cdots =u_m=v_1=0$. Thus
$$a_{[\vec u]^*, [v_1,\vec v]^*}^{[0_i, \vec s]^*}=\left\{\begin{array}{lll}
1, & [\vec u]=[0_m], [v_1,\vec v]=[0_{i-m}, \vec s];\\
0, &\text{otherwise}.
\end{array}\right.
$$

\noindent
\meqref{i:comp2}  If $m=i$, then Eq.\,\meqref{eq:0i} gives
$$\Big(\shap_{\le 0}^*([0_{i}, \vec s]^*)\Big)\Big([u_1,\cdots\hskip-1mm, u_i]\otimes [v_1,\vec v]\Big)
=(-1)^{u_1+\cdots+u_i}[0_{i},\vec s]^*([v_1+u_1+\cdots+u_i, \vec v]).
$$
So
$$a_{[\vec u]^*, [v_1,\vec v]^*}^{[0_i, \vec s]^*}=\left\{\begin{array}{lll}
(-1)^{u_1+\cdots+u_i}, &[v_1+u_1+\cdots+u_i,\vec v]=[\vec s];\\
0, &\text{otherwise}.
\end{array}\right.
$$

\noindent
\meqref{i:comp3}  If $m>i$, then by Eq.\,\meqref{eq:0i},
\begin{align*}
&\Big(\shap_{\le 0}^*([0_i, \vec s]^*)\Big)\Big([u_1,\cdots\hskip-1mm, u_i, \vec u\,']\otimes [v_1,\vec v]\Big)\\
=&(-1)^{u_1}[0_i, \vec s]^*\Big(\Big[0, [u_2,\cdots\hskip-1mm, u_i, \vec u\,']\shap_{\le 0} [v_1+u_1,\vec v]\Big]\Big)\\
=&(-1)^{u_1+\cdots+u_i}[0_i, \vec s]^*\Big(\Big[0_i,[\vec u\,']\shap_{\le 0} [v_1+u_1+\cdots+u_i, \vec v]\Big]\Big)\\
=&(-1)^{u_1+\cdots+u_i}a_{[\vec u\,']^*, [v_1+u_1+\cdots+u_{i}, \vec v]^*}^{[ \vec s]^*}
\end{align*}
giving Eq. (\mref{eq:cd}). This completes the proof.
\end{proof}

\begin {coro}
\mlabel{lemma:zeros}
For  $(0_k)\in \Z_{\le 0}^k$, we have
$\shap_{\le 0}^*([0_k]^*)=\sum\limits_{j=0}^{k}[0_j]^*\otimes [0_{k-j}]^*.$
\end{coro}

\subsubsection{Dual  coproduct on $\calhd$ induced by $\shap_{\le 0}^*$}
We next show that the family $\{\tilde{\delta}_i\,|\,i\in \Z_{\ge 0}\}$ is a shifted coderivation with respect to $\shap_{\le 0}^*$ that is similar to Theorem~\mref{pp:unique}. This property will lead to Theorem\,\mref{thm:deltaa=c} and Corollary\,\mref{coro:CoalgebraIso1}, providing the last ingredient in the proof of Theorem\,\mref{t:dhaiso}.
We begin with  $\tilde{\delta}_1$.

\begin{prop}
\mlabel{pp:tildep1commute}
The induced operator $\tilde{\delta}_1$  commutes with $\shap_{\le 0}^*$,
$$(\id\, \cks \tilde{\delta}_1+\tilde{\delta}_1 \otimes \id)\shap_{\le 0}^*=\shap_{\le 0}^*\tilde{\delta}_1.
$$
\end{prop}

\begin{proof}
It is directly checked that 
$$\Big((\id\, \cks \tilde{\delta}_1+\tilde{\delta}_1 \otimes \id)\shap_{\le 0}^*\Big)({\bf 1}^*)=0=\big(\shap_{\le 0}^* \tilde{\delta}_1\big)({\bf 1}^*).
$$
For the rest of the proof, we will follow the notations in Eq. (\mref{eq:a}), but suppress the subscripts of the sums for clarity. For $(s_1,\vec s)\in \Z_{\le 0}^{k}$, applying  Lemma \mref{lem:relationaaab}, the left-hand side of the equality becomes
\begin{align*}
&(\id \,\cks \tilde{\delta}_1+\tilde{\delta}_1 \otimes \id)\shap_{\le 0}^*([s_1,\vec s]^*)\\
=&(\id \, \cks \tilde{\delta}_1+\tilde{\delta}_1 \otimes \id)({\bf1}^*\otimes [s_1, \vec s]^*+[s_1,\vec s]^*\otimes{\bf 1}^*)
+\sum a_{[u_1, \vec u]^*, [\vec v]^*}^{[s_1, \vec s]^*}(\id \, \cks \tilde{\delta}_1+\tilde{\delta}_1 \otimes \id)([u_1,\vec u]^*\otimes [\vec v]^*)\\
=&(1-s_1)({\bf1}^*\otimes [s_1-1, \vec s]^*+[s_1-1, \vec s]^*\otimes{\bf1}^*)
+\sum (1-u_1)a_{[u_1, \vec u]^*, [\vec v]^*}^{[s_1,\vec s]^*}([u_1-1,\vec u]^*\otimes [\vec v]^*)\\
\overset{(\ref{e:relationaaab})}{=}&(1-s_1)({\bf1}^*\otimes [s_1-1, \vec s]^*+[s_1-1, \vec s]^*\otimes {\bf1}^*)
+ (1-s_1)\sum  a_{[u_1-1,\vec u]^*, [ \vec v]^*}^{[s_1-1, \vec s]^*}([u_1-1,\vec u]^*\otimes [\vec v]^*).
\end{align*}
For the right-hand side, we have
\begin{align*}
\shap_{\le 0}^*\tilde{\delta}_1([s_1,\vec s]^*)=&(1-s_1)\shap_{\le 0}^*([s_1-1,\vec s]^*)\\
=&(1-s_1)\Big({\bf1}^*\otimes [s_1-1, \vec s]^*\!+\![s_1\!-\!1, \vec s]^*\otimes {\bf1}^*\Big)
+(1-s_1)\!\!\sum a_{[u_1, \vec u]^*, [\vec  v]^*}^{[s_1-1, \vec s]^*}\Big([u_1,\vec u]^*\otimes [\vec v]^*\Big).
\end{align*}
Lemma \mref{lem:0foru1greats1} gives
$$a_{[0,\vec u]^*,[\vec v]^*}^{[s_1-1, \vec s]^*}=0.$$
So
$$\sum  a_{[u_1-1,\vec u]^*, [v_1, \vec v]^*}^{[s_1-1, \vec s]^*}\Big([u_1-1,\vec u]^*\otimes [ \vec v]^*\Big)=\sum a_{[u_1, \vec u]^*, [\vec  v]^*}^{[s_1-1, \vec s]^*}\Big([u_1,\vec u]^*\otimes [\vec v]^*\Big).
$$
Thus the two sides of the desired equality agree.
\end{proof}

\begin{prop}
\mlabel{thm:deltacommutative}
The family $\{\tilde{\delta}_{i}\,|\,i\in \Z_{\ge 0}\}$ is a shifted coderivation with respect to $\shap_{\le 0}^*:$
$$(\id \, \cks \tilde{\delta}_{i}+\tilde{\delta}_{i}\otimes \id)\shap_{\le 0}^*=\shap_{\le 0}^* \tilde{\delta}_{i}.
$$
\end{prop}

\begin{proof}
We prove the theorem  by  induction  on  $i\ge 1$.  It is true for $i=1$  by  Proposition  \mref{pp:tildep1commute}. For a given $i\geq 2$, assume that for  $1\le j\le i-1$, the equation
\begin{equation}\mlabel {eq:induhypo}
(\id\ \cks \tilde{\delta}_{j}+\tilde{\delta}_{j}\otimes \id)\shap_{\le 0}^*=\shap_{\le 0}^* \tilde{\delta}_{j}
\end{equation}
holds.  Now we consider the  case  of $j=i$. More precisely, for $(a_1,\cdots\hskip-1mm,a_{i-1}, s_1,\cdots\hskip-1mm,s_k)\in \Z_{\le 0}^{k+i-1}$ abbreviated to $(a_1,\cdots\hskip-1mm,a_{i-1}, s_1, \vec s )$, we need to prove
\begin{equation}
	\begin{split}
		&(\id\ \cks \tilde{\delta}_{i}+\tilde{\delta}_{i}\otimes \id)\Big(\shap_{\le 0}^*([a_1,\cdots\hskip-1mm,a_{i-1}, s_1,\vec s]^*)\Big)=\shap_{\le 0}^*\tilde{\delta}_{i}([a_1,\cdots\hskip-1mm,a_{i-1} , s_1,\vec s]^*).
	\end{split}
	\mlabel {eq:general}
\end{equation}
For this purpose, we apply another induction on $n:=-(a_1+\cdots+a_{i-1})$,
carried out in the base step and inductive step as follows.

\noindent
{\bf Base Step on $n\ge 0$:} We first prove the case of $n=0$, that is,
\begin {equation}
\mlabel {eq:zeros}
(\id\ \cks \tilde{\delta}_{i}+\tilde{\delta}_{i}\otimes \id)\Big(\shap_{\le 0}^*([0_{i-1}, s_1,\vec s]^*)\Big)=\shap_{\le 0}^*\tilde{\delta}_{i}([0_{i-1} , s_1,\vec s]^*).
\end{equation}
For the rest of the proof, we will follow the notations in Eq. (\mref{eq:a}), but suppress the subscripts of the sums for clarity. The left-hand side of Eq.~\meqref{eq:zeros} gives
\begin{align*}
&(\id \ \cks \tilde{\delta}_{i}+\tilde{\delta}_{i}\otimes \id)\Big(\shap_{\le 0}^*([0_{i-1}, s_1,\vec s]^*)\Big)\\
=&{\bf 1}^*\otimes \tilde{\delta}_i([0_{i-1}, s_1,\vec s]^*)+\tilde{\delta}_i([0_{i-1}, s_1,\vec s]^*)\otimes {\bf 1}^*\\
&+\sum a_{[\vec u]^*, [\vec v]^*}^{[0_{i-1}, s_1,\vec s]^*}\tilde{\delta}_{i}([\vec u]^*)\otimes [\vec v]^*+\sum a_{[\vec u]^*, [\vec v]^*}^{[0_{i-1}, s_1,\vec s]^*}(\id\ \cks \tilde{\delta}_{i})([\vec u]^*\otimes [\vec v]^*).
\end{align*}
Applying the induction hypothesis (\mref{eq:induhypo}), the right-hand side becomes
\begin{align*}
&\shap_{\le 0}^*\tilde{\delta}_{i}([0_{i-1} , s_1,\vec s]^*)\\
=&\shap_{\le 0}^*\tilde{\delta}_{i-1}([ 0_{i-1}, s_1,\vec s]^*)+(1-s_1) \shap_{\le 0}^*([0_{i-1},s_1-1,\vec s]^*)\\
=&(\id\ \cks \tilde{\delta}_{i-1}+\tilde{\delta}_{i-1}\otimes \id)\Big(\shap_{\le 0}^*([0_{i-1} , s_1,\vec s]^*)\Big)+(1-s_1)\shap_{\le 0}^*([0_{i-1},s_1-1,\vec s]^*)\\
=&{\bf1}^*\otimes \tilde{\delta}_{i-1}([0_{i-1} , s_1,\vec s]^*)+\tilde{\delta}_{i-1}([0_{i-1} , s_1,\vec s]^*)\otimes {\bf1}^*\\
&+\sum a_{[\vec u]^*, [\vec v]^*}^{[0_{i-1}, s_1,\vec s]^*}\tilde{\delta}_{i-1}([\vec u]^*)\otimes [\vec v]^*+\sum a_{[\vec u]^*, [\vec v]^*}^{[0_{i-1}, s_1,\vec s]^*}(\id\ \cks \tilde{\delta}_{i-1})([\vec u]^*\otimes [\vec v]^*)\\
&+(1-s_1)\Big({\bf1}^*\otimes [0_{i-1},s_1-1,\vec s]^*+[0_{i-1},s_1-1,\vec s]^*\otimes {\bf 1}^*\Big)+(1-s_1)\sum a_{[\vec u]^*, [\vec v]^*}^{[0_{i-1}, s_1-1,\vec s]^*}[\vec u]^*\otimes [\vec v]^*\\
=&{\bf 1}^*\otimes \tilde{\delta}_i([0_{i-1}, s_1,\vec s]^*)+\tilde{\delta}_i([0_{i-1}, s_1,\vec s]^*)\otimes {\bf 1}^*+\sum a_{[\vec u]^*, [\vec v]^*}^{[0_{i-1}, s_1,\vec s]^*}\tilde{\delta}_{i-1}([\vec u]^*)\otimes [\vec v]^*\\
&+\sum a_{[\vec u]^*, [\vec v]^*}^{[0_{i-1}, s_1,\vec s]^*}(\id\ \cks \tilde{\delta}_{i-1})([\vec u]^*\otimes [\vec v]^*)+(1-s_1)\sum a_{[\vec u]^*, [\vec v]^*}^{[0_{i-1}, s_1-1,\vec s]^*}[\vec u]^*\otimes [\vec v]^*.
\end{align*}
Hence the desired equality \meqref{eq:zeros}
is equivalent to
\begin{equation}
\begin{aligned}
&\sum a_{[\vec u]^*, [\vec v]^*}^{[0_{i-1}, s_1,\vec s]^*}\tilde{\delta}_{i}([\vec u]^*)\otimes [\vec v]^*+\sum a_{[\vec u]^*, [\vec v]^*}^{[0_{i-1}, s_1,\vec s]^*}(\id\ \cks \tilde{\delta}_{i})([\vec u]^*\otimes [\vec v]^*) \\
=&\sum a_{[\vec u]^*, [\vec v]^*}^{[0_{i-1}, s_1,\vec s]^*}\tilde{\delta}_{i-1}([\vec u]^*)\otimes [\vec v]^*+\sum a_{[\vec u]^*, [\vec v]^*}^{[0_{i-1}, s_1,\vec s]^*}(\id \ \cks \tilde{\delta}_{i-1})([\vec u]^*\otimes [\vec v]^*)\\
&+(1-s_1)\sum a_{[\vec u]^*, [\vec v]^*}^{[0_{i-1}, s_1-1,\vec s]^*}[\vec u]^*\otimes [\vec v]^*.
\end{aligned}
\mlabel{eq:prove}
\end{equation}
For $m\ge 1$, let us focus on the terms in Eq.\,(\mref{eq:prove}) for which the subindex $[\vec u]$ is of depth $m$. We distinguish three cases.

{\bf Case 1:} If $m<i-1$, then Eq.\,(\mref{eq:prove}) reduces to
\begin{equation}
\begin{split}
\sum a_{[\vec u]^*, [\vec v]^*}^{[0_{i-1}, s_1,\vec s]^*}\Big([\vec u]^*\otimes \tilde{\delta}_{i-m}([\vec v]^*)\Big)=&\sum  a_{[\vec u]^*, [\vec v]^*}^{[0_{i-1}, s_1,\vec s]^*}[\vec u]^*\otimes \tilde{\delta}_{i-1-m}([\vec v]^*) \\
&+(1-s_1)\sum a_{[\vec u]^*, [\vec v]^*}^{[0_{i-1}, s_1-1,\vec s]^*}[\vec u]^*\otimes [\vec v]^*.
\end{split}
\mlabel{eq:mlei}
\end{equation}
By  Lemma  \mref{lem:mcomparei}, the left-hand side  of Eq. (\mref{eq:mlei}) equals  to
$$
[0_m]^*\otimes \tilde{\delta}_{i-m}([0_{i-1-m},s_1,\vec s]^*).
$$
Similarly, the  right-hand side of  Eq.  (\mref{eq:mlei}) equals to
\begin{align*}
&[0_m]^*\otimes \tilde{\delta}_{i-1-m}([0_{i-1-m},s_1,\vec s]^*)+(1-s_1)[0_m]^*\otimes [0_{i-1-m},s_1-1,\vec s]^*.
\end{align*}
Then by the definition of $\tilde{\delta}_{i-m}$, Eq.  (\mref{eq:mlei}) holds.

{\bf Case 2:} If $m=i-1$, then denote $\vec u=(u_1,\cdots\hskip-1mm, u_m)$ and $\vec v=(v_1,\vec v\,')$.  By  Lemma  \mref{lem:mcomparei}, the coefficient $a_{[\vec u]^*, [\vec v]^*}^{[0_{i-1}, s_1,\vec s]^*}$ equals to $(-1)^{u_1+\cdots+u_m}$  if  $[v_1+u_1+\cdots+u_m, \vec v\,']=[s_1,\vec s]$, and equals to $0$ otherwise. So Eq. (\mref{eq:prove})  is  equivalent to
\begin{align*}
& \sum_{[\vec u]}(1\hskip -1mm +\hskip -1mm |\vec u |\hskip -1mm -\hskip -1mm s_1)(-1)^{|\vec u |}[\vec u]^*\otimes [s_1-1-|\vec u |, \vec s]^* \\
=&\sum_{[\vec u]} (-1)^{|\vec u |}\tilde{\delta}_{i-1}([\vec u]^*)\otimes [s_1-|\vec u|, \vec s]^*+(1-s_1)\sum_{[\vec u]}(-1)^{|\vec u |} [\vec u]^*\otimes [s_1-1-|\vec u |, \vec s]^*\\
=&\sum_{[\vec u]}\sum_{\ell=1}^{i-1}(1-u_{\ell}) (-1)^{|\vec u |}[u_1,\cdots \hskip -1mm , u_{\ell}\hskip -1mm - \hskip -1mm  1,\cdots\hskip-1mm, u_{i-1}]^*\otimes [s_1-|\vec u|, \vec s]^*\\
&+(1-s_1)\sum_{[\vec u]}(-1)^{|\vec u |} [\vec u]^*\otimes [s_1-1-|\vec u |, \vec s]^*.
\end{align*}
Here $|\vec u|=u_1+\cdots +u_m$. By  a  change of  variables, we obtain
\begin{align*}
&\sum_{[\vec u]}(1-u_{\ell}) (-1)^{|\vec u |}[u_1,\cdots\hskip-1mm, u_{\ell}-1,\cdots\hskip-1mm, u_{i-1}]^*\otimes [s_1-|\vec u|, \vec s]^*\\
=&\sum_{[\vec u], u_{\ell}\le -1} (-u_{\ell})(-1)^{|\vec u|+1} [u_1,\cdots\hskip-1mm, u_{\ell},\cdots\hskip-1mm, u_{i-1}]^*\otimes [s_1-|\vec u|-1, \vec s]^*\\
=&\sum_{[\vec u]}u_{\ell}(-1)^{|\vec u|} [u_1,\cdots\hskip-1mm, u_{\ell},\cdots\hskip-1mm, u_{i-1}]^*\otimes [s_1-|\vec u|-1, \vec s]^*.
\end{align*}
So  Eq.  (\mref{eq:prove})  holds for $m=i-1$.

{\bf Case 3:} If $m\ge i$,  then Eq. (\mref{eq:prove}) is reduced to
\small{
\begin{equation}
\begin{split}
&\sum a_{[\vec u]^*, [\vec v]^*}^{[0_{i-1}, s_1,\vec s]^*}\tilde{\delta}_{i}([\vec u]^*)\otimes [\vec v]^* =\sum a_{[\vec u]^*, [\vec v]^*}^{[0_{i-1}, s_1,\vec s]^*}\tilde{\delta}_{i-1}([\vec u]^*)\otimes [\vec v]^*+(1-s_1)\sum a_{[\vec u]^*, [\vec v]^*}^{[0_{i-1}, s_1-1,\vec s]^*}[\vec u]^*\otimes [\vec v]^*.
\end{split}
\mlabel{eq:mgei}
\end{equation}}
Denote  $\vec u=(u_1,\cdots\hskip-1mm,u_{i-1}, u_i,\vec u\,')$ and $\vec v=(v_1, \vec v\,')$. Since
\begin{align*}
&\sum a_{[\vec u]^*, [\vec v]^*}^{[0_{i-1}, s_1,\vec s]^*}\tilde{\delta}_{i}([\vec u]^*)\otimes [\vec v]^*\\
=&\sum a_{[\vec u]^*, [\vec v]^*}^{[0_{i-1}, s_1,\vec s]^*}\tilde{\delta}_{i-1}([\vec u]^*)\otimes [\vec v]^*+\sum (1-u_i)a_{[\vec u]^*, [\vec v]^*}^{[0_{i-1}, s_1,\vec s]^*}[u_1,\cdots\hskip-1mm,u_{i-1}, u_i-1,\vec u\,']^*\otimes [\vec v]^*,
\end{align*}
Eq.  (\mref{eq:mgei})  can be reduced  to
\begin{align*}
&\sum (1-u_i)a_{[u_1,\cdots\hskip-1mm,u_{i-1}, u_i,\vec u\,']^*, [\vec v]^*}^{[0_{i-1}, s_1,\vec s]^*}[u_1,\cdots\hskip-1mm,u_{i-1}, u_i-1,\vec u\,']^*\otimes [\vec v]^*\\
=&(1-s_1)\sum a_{[u_1,\cdots\hskip-1mm,u_{i-1}, u_i,\vec u\,']^*, [\vec v]^*}^{[0_{i-1}, s_1-1,\vec s]^*}[u_1,\cdots\hskip-1mm,u_{i-1}, u_i,\vec u\,']^*\otimes [\vec v]^*.
\end{align*}
By  a  change of  variables,  it  reduces to
\begin{align*}
&\sum (-u_i)a_{[u_1,\cdots\hskip-1mm,u_{i-1}, u_i+1,\vec u\,']^*, [\vec v]^*}^{[0_{i-1}, s_1,\vec s]^*}[u_1,\cdots\hskip-1mm,u_{i-1}, u_i,\vec u\,']^*\otimes [\vec v]^*\\
=&(1-s_1)\sum a_{[u_1,\cdots\hskip-1mm,u_{i-1}, u_i,\vec u\,']^*, [\vec v]^*}^{[0_{i-1}, s_1-1,\vec s]^*}[u_1,\cdots\hskip-1mm,u_{i-1}, u_i,\vec u\,']^*\otimes [\vec v]^*.
\end{align*}
By  Lemma  \mref{lem:mcomparei}, we  only  need to prove
$$(-u_i)a_{[u_i+1,\vec u\,']^*,[v_1+u_1+\cdots+u_{i-1}, \vec v\,']^*}^{[s_1,\cdots\hskip-1mm,s_k]^*}=(1-s_1)a_{[u_i,\vec u\,']^*, [v_1+u_1+\cdots+u_{i-1}, \vec v\,']^*}^{[s_1-1,\cdots\hskip-1mm,s_k]^*},
$$
which  is  true  by  Lemma  \mref{lem:relationaaab}.

Thus we obtain Eq.~\meqref{eq:prove} and thus Eq. (\mref {eq:zeros}). This completes the base step on the induction on $n\ge 0$.

\smallskip

\noindent
{\bf Inductive Step on $n\ge 0$:} For a given $\ell\ge 0$, assume that for $[ a_1,\cdots\hskip-1mm,a_{i-1}, s_1, \vec s ]$ with $n=-(a_1+\cdots +a_{i-1})\le \ell\le 0$, Eq. (\mref {eq:general}) holds, that is, for $a_1+\cdots +a_{i-1}\le \ell$,
\begin{equation}
(\id\ \cks \tilde{\delta}_{i}+\tilde{\delta}_{i}\otimes \id)\shap_{\le 0}^*([a_1,\cdots\hskip-1mm , a_{i-1}, s_1, \vec s]^*)=\shap_{\le 0}^* \tilde{\delta}_{i}([a_1,\cdots\hskip-1mm, a_{i-1}, s_1, \vec s]^*).\label{eq:iellinduction}
\end{equation}
Consider $[ a_1,\cdots\hskip-1mm,a_{i-1}, s_1, \vec s ]$ with  $n=-(a_1+\cdots+a_{i-1})=\ell+1$. Since $\ell+1\le -1 $, there is $r$ with $1\le r\le i-1$ such that $a_r<0$. Let $r_0$ be the least index such that  $a_{r_0}< 0$. Then $a_{r_0}+1\le 0$. So $[a_1,\cdots\hskip-1mm ,a_{r_0}+1,\cdots\hskip-1mm, a_{i-1},s_1,\vec s]^*$ is still in $\calhf$. By the definition of $\tilde{p}_{r_0}$ and Eq.\,(\ref{eq:tildedelta}), we obtain
\begin{equation*}
\begin{split}
&\tilde{p}_{r_0}([a_1,\cdots\hskip-1mm ,a_{r_0}+1,\cdots\hskip-1mm, a_{i-1},s_1,\vec s]^*)=(\tilde{\delta}_{r_0}-\tilde{\delta}_{r_0-1})([a_1,\cdots\hskip-1mm ,a_{r_0}+1,\cdots\hskip-1mm, a_{i-1},s_1,\vec s]^*)\\
=&\sum_{k=1}^{r_0-1}(1-a_k)[a_1,\cdots\hskip-1mm ,a_{k}-1,\cdots\hskip-1mm, a_{i-1},s_1,\vec s]^*+\big(1-(a_{r_0}+1)\big)[a_1,\cdots\hskip-1mm ,a_{r_0},\cdots\hskip-1mm, a_{i-1},s_1,\vec s]^*\\
&-\sum_{k=1}^{r_0-1}(1-a_k)[a_1,\cdots\hskip-1mm ,a_{k}-1,\cdots\hskip-1mm, a_{i-1},s_1,\vec s]^*\\
=&-a_{r_0}[a_1,\cdots\hskip-1mm ,a_{r_0},\cdots\hskip-1mm, a_{i-1},s_1,\vec s]^*.
\end{split}
\end{equation*}
Hence
\begin{equation}
\mlabel{eq:tildepdef}
\frac{\tilde{p}_{r_0}}{-a_{r_0}}([a_1,\cdots\hskip-1mm ,a_{r_0}+1,\cdots\hskip-1mm, a_{i-1},s_1,\vec s]^*)=[a_1,\cdots\hskip-1mm ,a_{r_0},\cdots\hskip-1mm, a_{i-1},s_1,\vec s]^*.
\end{equation}
This gives
\begin{equation*}
\begin{split}
&(\id\ \cks \tilde{\delta}_{i}+\tilde{\delta}_{i}\otimes \id)\Big(\shap_{\le 0}^*([a_1,\cdots\hskip-1mm,a_{r_0},\cdots\hskip-1mm, a_{i-1}, s_1,\vec s]^*)\Big)\\
=&(\id \ \cks \tilde{\delta}_{i}+\tilde{\delta}_{i}\otimes \id)\Big(\shap_{\le 0}^*\frac{\tilde{p}_{r_0}}{-a_{r_0}}([a_1,\cdots\hskip-1mm,a_{r_0}+1,\cdots\hskip-1mm, a_{i-1}, s_1,\vec s]^*)\Big).
\end{split}
\end{equation*}
Applying the inductive hypothesis on $j$ (stated in Eq.\, (\ref{eq:induhypo})) to $j=r_0-1$ and $j=r_0$, for $[\vec t]\in \calhf$, we obtain
$$(\id\ \cks \tilde{\delta}_{r_0}+\tilde{\delta}_{r_0}\otimes \id)\shap_{\le 0}^*([\vec t])=\shap_{\le 0}^* \tilde{\delta}_{r_0}([\vec t]),\,(\id\ \cks \tilde{\delta}_{r_0-1}+\tilde{\delta}_{r_0-1}\otimes \id)\shap_{\le 0}^*([\vec t])=\shap_{\le 0}^* \tilde{\delta}_{r_0-1}([\vec t]).
$$
Then we get
\begin{equation}
(\id\ \cks \tilde{p}_{r_0}+\tilde{p}_{r_0}\otimes \id)\shap_{\le 0}^*([\vec t])=\shap_{\le 0}^* \tilde{p}_{r_0}([\vec t]), \quad [\vec t]\in\calhf.
\mlabel{eq:pjcommuteshapdual}
\end{equation}
Hence the equation above
\begin{equation*}
\begin{split}
&(\id \ \cks \tilde{\delta}_{i}+\tilde{\delta}_{i}\otimes \id)\Big(\shap_{\le 0}^*\frac{\tilde{p}_{r_0}}{-a_{r_0}}([a_1,\cdots\hskip-1mm,a_{r_0}+1,\cdots\hskip-1mm, a_{i-1}, s_1,\vec s]^*)\Big)\\
\overset{\textcircled{1}}{=}&\frac{(\id \ \cks \tilde{\delta}_{i}+\tilde{\delta}_{i}\otimes \id)(\id \ \cks \tilde{p}_{r_0}+\tilde{p}_{r_0}\otimes \id)}{-a_{r_0}}\shap_{\le 0}^*([a_1,\cdots\hskip-1mm, a_{r_0}+1,\cdots\hskip-1mm, a_{i-1}, s_1,\vec s]^*)\\
\overset{\textcircled{2}}{=}&\frac{(\id \ \cks \tilde{p}_{r_0}+\tilde{p}_{r_0}\otimes \id)(\id\ \cks \tilde{\delta}_{i}+\tilde{\delta}_{i}\otimes \id)}{-a_{r_0}}\shap_{\le 0}^*([a_1,\cdots\hskip-1mm, a_{r_0}+1,\cdots\hskip-1mm, a_{i-1}, s_1, \vec s]^*\\
\overset{\textcircled{3}}{=}&\frac{\id\ \cks \tilde{p}_{r_0}+\tilde{p}_{r_0}\otimes \id}{-a_{r_0}}\shap_{\le 0}^*\Big(\tilde{\delta}_{i}([a_1,\cdots\hskip-1mm, a_{r_0}+1,\cdots\hskip-1mm, a_{i-1},s_1, \vec s]^*)\Big)\\
\overset{\textcircled{4}}{=}&\shap_{\le 0}^*\Big(\frac{\tilde{p}_{r_0}}{-a_{r_0}}\tilde{\delta}_{i}([a_1,\cdots\hskip-1mm,a_{r_0}+1,\cdots\hskip-1mm, a_{i-1}, s_1,\vec s]^*)\Big)\\
\overset{\textcircled{5}}{=}&\shap_{\le 0}^*\Big(\tilde{\delta}_{i}([a_1,\cdots\hskip-1mm,a_{r_0},\cdots \hskip-1mm, a_{i-1}, s_1,\vec s]^*)\Big).
\end{split}
\end{equation*}
Here $\textcircled{1}$ followed by Eq.\,(\mref{eq:pjcommuteshapdual}) for $[\vec t]=[a_1,\cdots\hskip-1mm, a_{r_0}+1,\cdots\hskip-1mm, a_{i-1}, s_1, \vec s]^*$, $\textcircled{2}$ followed by Lemma \mref{lem:p1phieqphipi} Item(\mref{eq:tildepcomm}), $\textcircled{3}$ followed by Eq.\,(\ref{eq:iellinduction}) for the case $a_1+\cdots+(a_{r_0}+1)+\cdots+a_{i-1}=\ell$, $\textcircled{4}$ followed by Eq.\,(\mref{eq:pjcommuteshapdual}) for $[\vec t]=\tilde{\delta}_{i}([a_1,\cdots\hskip-1mm,a_{r_0},\cdots \hskip-1mm, a_{i-1}, s_1,\vec s]^*)$, $\textcircled{5}$ followed by Lemma \mref{lem:p1phieqphipi}.(\mref{eq:tildepcomm}) and Eq.\,(\mref{eq:tildepdef}).
This completes the induction on $n\ge 0$, thereby the induction on $i\ge 1$, and therefore the proof of the proposition.
\end{proof}

Now we define a coproduct
$\Deltac: \calhd\longrightarrow \calhd\otimes \calhd$
by the \rdual $\varphi$:
\begin {equation}
\mlabel {eq:Delta}
\Deltac:=(\varphi^{-1}\otimes \varphi^{-1})\shap_{\le 0}^{\ast}\varphi.
\end{equation}

\begin{prop}
\mlabel{prop:deltacommuteDeltale0}
The family $\{\delta_{i}\,|\,i\in \Z_{\ge 0}\}$ is a shifted coderivation with respect to $\Deltac:$
$$(\id\ \cks \delta_i+\delta_i\otimes \id)\Deltac =\Deltac \delta_i.
$$
\end{prop}

\begin{proof}
By the definition of $\Deltac$,  Lemma  \mref{lem:p1phieqphipi} and Proposition  \mref{thm:deltacommutative}, we obtain
\begin{align*}
(\id\ \cks \delta_i+\delta_i\otimes \id)\Deltac&=(\id \ \cks \delta_i+\delta_i\otimes \id)(\varphi^{-1}\otimes \varphi^{-1})\shap_{\le 0}^*\varphi\\
&=(\varphi^{-1}\otimes \varphi^{-1})(\id \ \cks \tilde{\delta}_i+\tilde{\delta}_i\otimes \id)\shap_{\le 0}^*\varphi\\
&=(\varphi^{-1}\otimes \varphi^{-1})\shap_{\le 0}^*\tilde{\delta}_i\varphi\\
&=(\varphi^{-1}\otimes \varphi^{-1})\shap_{\le 0}^*\varphi \delta_i\\
&=\Deltac\delta_i. \qedhere
\end{align*}
\end{proof}

\begin{theorem}
\mlabel{thm:deltaa=c}
 On $\calhd$ we have $\Deltaa=\Deltac$.
\end{theorem}

\begin{proof} To prove this result, we only need to check that $\Deltac$ satisfies the conditions in Proposition \mref{pp:unique} as follows.
\begin{enumerate}
  \item By definition, $\Deltac ({\bf 1}))={\bf 1}\otimes {\bf 1}$.
  \item For  $(1_k) \in \Z_{\ge 1}^k$, Corollary \mref {lemma:zeros} gives $\Deltac([1_k])=\sum_{j=0}^{k}[1_j]\otimes [1_{k-j}].$
  \item By Proposition  \mref{prop:deltacommuteDeltale0}, we have
$(\id\ \cks \delta_i+\delta_i\otimes \id)\Deltac =\Deltac \delta_i.$
\end{enumerate}
Therefore, the uniqueness from Proposition \mref{pp:unique} gives the conclusion.
\end{proof}

\begin{coro}
\mlabel{coro:CoalgebraIso1}
The \rdual $\varphi$  is an isomorphism of coalgebras $$\varphi: (\calhd, \Deltaa, \varepsilon_{\ge 1})\to (\calhf, \shap_{\le 0}^*, u_{\le 0}^*).$$
\end{coro}
\begin{proof}
By Eq.\,\meqref{eq:Delta} and Theorem\,\mref{thm:deltaa=c} we have $(\varphi\otimes \varphi)\Deltaa=\shap_{\le 0}^*\varphi.$
Then we only need to check that $\varphi$ preserves the counit. By Eq.\,(\mref{eq:counitle0dual}), for a basis element $x\in\calhd$, we have
$$ u_{\le 0}^*\varphi(x)=\left\{\begin{array}{ccc}
1,& x={\bf1};\\
0, & \text{otherwise}.
\end{array}\right.
$$
So $u_{\le 0}^*\varphi=\varepsilon_{\ge 1}$, as needed.
\end{proof}

Taking the graded linear dual yields
\begin{coro}
\mlabel{coro:algebraIso1}
The map $\varphi^*: (\calhc, \shap_{\le 0}, u_{\le 0})\to (\calhg, \Deltae, \varepsilon_{\ge 1}^*)$ is an isomorphism for algebras.
\end{coro}

\noindent
{\bf Acknowledgments.} This research is supported by NSFC (12471062).

\noindent
{\bf Declaration of interests. } The authors have no conflict of interest to declare.

\noindent
{\bf Data availability. } Data sharing is not applicable
as no data were created or analyzed.

\end{document}